\documentclass[11pt]{amsart}

\usepackage{a4wide, amsmath, amsfonts, amssymb, amsthm, bbm, stmaryrd, breqn, enumerate}
\usepackage[hyperfootnotes=false]{hyperref}

\allowdisplaybreaks[2]

\numberwithin{equation}{section}

\newtheorem{theorem}{Theorem}[section]
\newtheorem{lemma}[theorem]{Lemma}
\newtheorem{corollary}[theorem]{Corollary}
\newtheorem{remark}[theorem]{Remark}
\newtheorem{proposition}[theorem]{Proposition}
\newtheorem{definition}[theorem]{Definition}

\newcommand{\dd}{\,\mathrm{d}}
\renewcommand{\d}{\mathrm{d}}
\newcommand{\D}{\mathrm{D}}
\newcommand{\X}{{\bf X}}
\newcommand{\cadlag}{c{\`a}dl{\`a}g{ }}
\renewcommand{\epsilon}{\varepsilon}
\newcommand{\R}{\mathbb{R}}

\newcommand{\1}{\mathbf{1}}

\newcommand{\bbX}{\mathbb{X}}
\newcommand{\tA}{\tilde{A}}
\newcommand{\tY}{\tilde{Y}}
\newcommand{\tX}{\tilde{X}}
\newcommand{\tbX}{\tilde{\mathbf{X}}}
\newcommand{\tbbX}{\tilde{\mathbb{X}}}
\newcommand{\tL}{\tilde{L}}
\newcommand{\tK}{\tilde{K}}
\newcommand{\tZ}{\tilde{Z}}
\newcommand{\cB}{\mathcal{B}}
\newcommand{\cL}{\mathcal{L}}
\newcommand{\cM}{\mathcal{M}}
\newcommand{\cP}{\mathcal{P}}
\newcommand{\cS}{\mathcal{S}}
\newcommand{\bX}{\mathbf{X}}

\newcommand{\ver}[1]{{\vert\kern-0.25ex\vert\kern-0.25ex\vert #1 \vert\kern-0.25ex\vert\kern-0.25ex\vert}}

\newcommand{\ty}{\tilde{y}}
\newcommand{\cV}{\mathcal{V}}

\title[C{\`a}dl{\`a}g rough differential equations with reflecting barriers]{C{\`a}dl{\`a}g rough differential equations\\ with reflecting barriers}

\author[Allan]{Andrew L. Allan}
\address{Andrew L. Allan, ETH Zurich, Switzerland}
\email{andrew.allan@math.ethz.ch}

\author[Liu]{Chong Liu}
\address{Chong Liu, University of Oxford, United Kingdom}
\email{chong.liu@maths.ox.ac.uk}

\author[Pr{\"o}mel]{David J. Pr{\"o}mel}
\address{David J. Pr{\"o}mel, University of Mannheim, Germany}
\email{proemel@uni-mannheim.de}

\date{\today}

\begin{document}

\begin{abstract}
  We investigate rough differential equations with a time-dependent reflecting lower barrier, where both the driving (rough) path and the barrier itself may have jumps. Assuming the driving signals allow for Young integration, we provide existence, uniqueness and stability results. When the driving signal is a c{\`a}dl{\`a}g $p$-rough path for $p \in [2,3)$, we establish existence to general reflected rough differential equations, as well as uniqueness in the one-dimensional case.
\end{abstract}

\maketitle




\section{Introduction}

Stochastic differential equations (SDEs) with reflecting barriers or boundary conditions have a long history in probability theory going back to Skorokhod~\cite{Skorohod1961}. Since the early works \cite{Skorohod1961,McKean1963,Watanabe1971,ElKaroui1975,Tanaka1979} regarding reflected diffusions in a half-space, there has been a considerable effort to deal with various generalizations, such as more intricate boundary conditions (see e.g. \cite{Lions1984,Saisho1987}) or more complex stochastic processes, like fractional Brownian motion (see e.g. \cite{Ferrante2013}) and general semimartingales (see e.g. \cite{Menaldi1985,Laukajtys2013,Falkowski2017}). Associated properties of these reflected diffusions have been widely studied as well, such as approximation results and support theorems, see e.g.~\cite{Laukajtys2003,Aida2013,Ren2016}. The theoretical study of reflected SDEs and of the closely related Skorokhod problem has been additionally motivated by their many applications, such as in queuing theory and statistical physics, see e.g.~\cite{Mandelbaum1995,Burdzy2002}.

A fresh perspective on stochastic differential equations was initiated by Lyons, providing a pathwise analysis of SDEs, first using Young integration~\cite{Lyons1994}, and then by introducing the theory of rough paths~\cite{Lyons1998}, which allows one to treat various random noises, such as fractional Brownian motion and continuous semimartingales. The rough path approach to stochastic differential equations has been celebrated for offering many advantages and new insights; for an overview see for instance the introductory textbook~\cite{Friz2014}. A first pathwise analysis of reflected SDEs was presented by~\cite{Ferrante2013,Falkowski2015} using Young integration, and by \cite{Besalu2014,Aida2015} using the more powerful theory of rough paths.

The aim of the present work is to provide a pathwise analysis of differential equations reflected at a c{\`a}dl{\`a}g time-dependent barrier $L \colon [0,T] \to \R^n$ of finite $p$-variation. More precisely, for a c{\`a}dl{\`a}g path $A \colon [0,T] \to \R^d$ of finite $q$-variation and a c{\`a}dl{\`a}g path $X \colon [0,T] \to \R^d$ of finite $p$-variation with $q \in [1,2)$ and $p \in [2,3)$ satisfying $1/q+1/p>1$, we study the differential equation
\begin{equation}\label{eq:Intro}
  \d Y_t = f_1(Y_t) \dd A_t + f_2(Y_t) \dd X_t + \d K_t,\quad t\in [0,T],
\end{equation}
where the solution $Y$ is reflected at the time-dependent barrier~$L$, that is, $Y_t^i \geq L_t^i$ for $i=1,\dots,n$, and the reflector term $K \colon [0,T] \to \R^n$ is a non-decreasing process fulfilling a standard minimality condition.

In the first part of the paper we suppose that the second vector field is trivial, i.e.~$f_2 = 1$. In this case classical Young integration~\cite{Young1936} is sufficient to define the remaining integral in~\eqref{eq:Intro}, and we can thus treat \eqref{eq:Intro} as a reflected Young differential equation. Under standard assumptions on the vector field, we show the existence and uniqueness of a solution to~\eqref{eq:Intro} using a Banach fixed point argument. Moreover, we prove that the solution map $(A,X) \mapsto (Y,K)$ is locally Lipschitz continuous with respect to both the $p$-variation distance and to the Skorokhod~$J_1$ $p$-variation distance. These results provide a comprehensive pathwise analysis of reflected Young differential equations. In particular, our results complement the existing literature (cf.~\cite{Ferrante2013,Aida2015,Falkowski2015}) in terms of the pathwise stability of the solution map, which constitutes one of the central advantages of a pathwise analysis of SDEs. For instance, pathwise stability results allow one to prove support and approximation results, as well as large deviation principles for stochastic differential equations, cf.~\cite{Friz2010}.

In the second part we consider general vector fields $f_1$ and $f_2$. In this case Young integration is no longer sufficient. We therefore assume that $X$ is a c{\`a}dl{\`a}g $p$-rough path in order to define the second integral in \eqref{eq:Intro} as a rough integral, turning \eqref{eq:Intro} into a reflected rough differential equation (reflected RDE). For this purpose we rely on the c{\`a}dl{\`a}g rough path theory of forward integration recently introduced in \cite{Friz2017,Friz2018}, a generalization of the now classical theory of continuous rough paths which also allows processes with jumps. Indeed, general semimartingales can be lifted to c{\`a}dl{\`a}g rough paths, as well as many other stochastic processes, such as suitable Gaussian processes, Dirichlet processes and Markov processes, see \cite{Friz2018,Chevyrev2019,Liu2018}. Hence, a c{\`a}dl{\`a}g rough path approach to \eqref{eq:Intro} significantly enlarges the class of well-posed reflected SDEs. As already pointed out in \cite{Aida2015} and \cite{Deya2019}, reflected rough differential equations face significant new challenges compared to the treatment of classical RDEs, the main reason being the lack of regularity of the Skorokhod map, particularly its lack of Lipschitz continuity of the space of controlled paths (see Section~\ref{sec:rough existence} below).

We establish the existence of a solution to the reflected RDE~\eqref{eq:Intro} based on Schauder's fixed point theorem and $p$-variation estimates for the Skorokhod map due to \cite{Falkowski2015}. While Schauder's fixed point theorem is a classical argument in the context of differential equations, the present setting allowing driving signals $A$ and $X$ with jumps requires careful analysis, particularly in the introduction of a suitable compact set on the space of c{\`a}dl{\`a}g controlled paths. So far existence results for reflected RDEs are only known for continuous driving signals, see \cite{Besalu2014,Aida2015,Deya2019,Richard2020}. Similar results have been obtained in the context of sweeping processes with pathwise perturbations~\cite{Castaing2017,Falkowski2015b} and path-dependent rough differential equations \cite{Aida2016,Ananova2020}, both also covering reflected RDEs.

We then prove the uniqueness of the solution to the reflected rough differential equation~\eqref{eq:Intro} in the one-dimensional case, that is, when the solution $Y$ is real-valued. For multidimensional reflected RDEs a general uniqueness result cannot hold, as observed by Gassiat~\cite{Gassiat2021}, who provides a linear RDE in $n=2$ dimensions reflected at $0$ with infinitely many solutions. For one-dimensional reflected RDEs driven by continuous rough paths uniqueness was obtained by \cite{Deya2019} in the case $L=0$ and by \cite{Richard2020} in the case of time-dependent barriers~$L$ . The approach of \cite{Deya2019} (as well as \cite{Richard2020}) relies fundamentally on the sewing lemma and the rough Gr{\"o}nwall inequality of \cite{Deya2019b}, for which the continuity of the driving paths seems to be crucial, see Remark~\ref{rem:Deya approach}. Therefore, in order to treat the c{\`a}dl{\`a}g setting, our proof of uniqueness utilizes a novel approach based on a contradiction argument. Remarkably, this proof is rather transparent and is surprisingly short, particularly in the special case of continuous driving paths.

\medskip

\noindent\textbf{Organization of the paper:} In Section~\ref{sec:young setting} we provide existence, uniqueness and stability results for reflected differential equations driven by signals allowing for Young integration. In Section~\ref{sec:rough existence} we prove the existence of solutions to multidimensional reflected rough differential equations. Finally, we provide a uniqueness result for one-dimensional reflected RDEs in Section~\ref{sec:rough uniqueness}.

\medskip

\noindent\textbf{Acknowledgment:} A.~L.~Allan gratefully acknowledges financial support by the Swiss National Science Foundation via Project 200021\textunderscore 184647. C.~Liu gratefully acknowledges support from the Early Postdoc.Mobility Fellowship (No.~P2EZP2\textunderscore 188068) of the Swiss National Science Foundation, and from the G.~H.~Hardy Junior Research Fellowship in Mathematics awarded by New College, Oxford.

\subsection{Basic notation}

Let us start by introducing some standard definitions and notation used throughout the paper. 

\medskip

A partition~$\mathcal{P}=\mathcal{P}([s,t])$ of the interval~$[s,t]$ is a set of essentially disjoint intervals covering $[s,t]$, i.e.~$\mathcal{P} = \{[u_i,u_{i+1}]\,:\, s = u_0 < u_1 < \cdots < u_n = t\}$. The mesh size of a partition $\mathcal{P}$ is given by $|\mathcal{P}| := \max\{|u_{i+1} - u_i|\,:\,i=0,\dots,n-1\}$. Given a metric space $(E,d)$, the set $D([0,T];E)$ denotes the space of all c{\`a}dl{\`a}g (right-continuous with left-limits) paths from $[0,T]$ into $E$. For $p \geq 1$, the $p$-variation of $X \in D([0,T];E)$ over the interval $[s,t]$ is defined by
\begin{equation*}
  \|X\|_{p,[s,t]} := \bigg(\sup_{\mathcal{P}\subset[s,t]} \sum_{[u,v]\in \mathcal{P}} d(X_u,X_v)^p \bigg)^{\frac{1}{p}},
\end{equation*}
where the supremum is taken over all finite partitions $\mathcal{P}$ of the interval $[s,t]$, and the sum denotes the summation over all intervals $[u,v] \in \mathcal{P}$. Recall that, for every $s\in [0,T]$, the function $[s,T]\ni t\mapsto \|X\|_{p,[s,t]}$ is non-decreasing and right-continuous with $\|X\|_{p,[s,s]}:=\lim_{t\downarrow s} \|X\|_{p,[s,t]}=0$ (see \cite[Lemma~7.1]{Friz2018}), and the function $(s,t)\mapsto  \|X\|_{p,[s,t]}^p$ is superadditive, i.e.~$\|X\|_{p,[s,u]}^p+\|X\|_{p,[u,t]}^p \leq \|X\|_{p,[s,t]}^p$ for $0\leq s\leq u\leq t\leq T$. See \cite[Section~2.2]{Chistyakov1998} for these and further properties of $p$-variation seminorms. Moreover, we set\footnote{One can similarly define $\|X\|_{p,(s,t]}$, but since all the paths we consider are c{\`a}dl{\`a}g, this always coincides with $\|X\|_{p,[s,t]}$.}
\begin{equation*}
  \|X\|_{p,[s,t)} := \sup_{u < t} \|X\|_{p,[s,u]}.
\end{equation*}
A path $X \in D([0,T];E)$ is said to have finite $p$-variation for some $p \in [1,\infty)$ if $\|X\|_{p,[0,T]} < \infty$. We will denote the space of all c{\`a}dl{\`a}g paths of finite $p$-variation by $D^{p}([0,T];E)$. 

The space $\R^n$ is equipped with the Euclidean norm $|\,\cdot\,|$. For two real numbers $x,y \in \R$ we set $x \wedge y := \min \{x,y\}$ and $x\vee y := \max\{x,y\}$, and we write the positive part of a vector $x=(x^1,\dots,x^n) \in \R^n$ as
\begin{equation}\label{eq:positive part}
  [x]^+ := ([x^1]^+,\dots, [x^n]^+) \qquad \text{where} \qquad [x^i]^+ := x^i \vee 0.
\end{equation}
For two paths $X = (X^1,\dots,X^n)\in D([0,T];\R^n)$ and $Y = (Y^1,\dots,Y^n) \in D([0,T];\R^n)$ we write $X \leq Y$ to mean that $X^i \leq Y^i$ for every $i = 1,\dots,n$.

\medskip

Whenever $X \in D([0,T];B)$ takes values in a Banach space $(B,\|\hspace{0.5pt}\cdot\hspace{0.5pt}\|)$, we will write $\|X\|_{\infty} := \sup_{t \in [0,T]} \|X_t\|$ for the supremum norm and we will use the abbreviations
\begin{equation*}
  X_{s,t} := X_t - X_s, \qquad X_{t-}:=\lim_{s \to t,\, s<t} X_s \qquad \text{and} \qquad \Delta X_t := X_{t-,t} = X_t - X_{t-}.
\end{equation*}

The space of linear maps from $\R^d \to \R^n$ is denoted by $\mathcal{L}(\R^d;\R^n)$ and we write $C_b^k = C_b^k(\R^n;\mathcal{L}(\R^d;\R^n))$ for the space of $k$-times differentiable (in the Fr{\'e}chet sense) functions $f \colon \R^n \to \mathcal{L}(\R^d;\R^n)$ such that $f$ and all its derivatives up to order~$k$ are continuous and bounded. We equip this space with the norm
\begin{equation*}
  \|f\|_{C^k_b} := \|f\|_\infty + \|\D f\|_\infty + \cdots + \|\D^kf\|_\infty,
\end{equation*}
where $\|\hspace{0.5pt}\cdot\hspace{0.5pt}\|_{\infty}$ denotes the supremum norm on the corresponding spaces of operators.

\medskip

Let $(B,\|\cdot\|)$ be a normed space and $f, g \colon B \to \R$ two functions. We shall write $f \lesssim g$ or $f \leq Cg$ to mean that there exists a constant $C > 0$ such that $f(x) \leq C g(x)$ for all $x \in B$. Note that the value of such constants may change from line to line, and that the constants may depend on the normed space, e.g.~through its dimension or regularity parameters. If we want to emphasize the dependence of the constant $C$ on some particular variables $\alpha_1,\dots,\alpha_n$, then we will write $C = C_{\alpha_1,\dots,\alpha_n}$.

\section{Reflected Young differential equations}\label{sec:young setting}

In this section we shall study reflected differential equations driven by paths $A \colon [0,T] \to \R^d$ with sufficiently regularity to allow for Young integration. More precisely, we assume that $A \in D^q([0,T];\R^d)$, $X \in D^p([0,T];\R^n)$ and $L \in D^p([0,T];\R^n)$ with $q \in [1,2)$ and $p \geq q$ such that $1/p + 1/q > 1$. Given $f \in C_b^2$ and $y \in \R^n$, we seek for two paths $Y \in D^p([0,T];\R^n)$ and $K \in D^1([0,T];\R^n)$ satisfying the \emph{reflected Young differential equation}
\begin{equation}\label{eq:reflected Young equation}
  Y_t = y + \int_0^t f(Y_s)\dd A_s+ X_t + K_t, \qquad t \in [0,T],
\end{equation}
such that, for every $i=1,\dots,n$,
\begin{enumerate}[(a)]
  \item $Y_t^i \geq L^i_t$ for all $t \in [0,T]$,
  \item $K^i \colon [0,T] \to \R$ is a non-decreasing function such that $K^i_0 = 0$, and
  \begin{equation}\label{eq:K is minimal}
    \int_0^t (Y^i_s-L^i_s)\dd K^i_s = 0, \qquad t \in [0,T],
  \end{equation}
  where the integral in \eqref{eq:K is minimal} is understood in the Lebesgue--Stieltjes sense.
\end{enumerate}
In the reflected Young differential equation~\eqref{eq:reflected Young equation} the integral $ \int_0^t f(Y_s)\dd A_s$ is well-defined as a Young integral, in the sense of~\cite{Young1936}; see also \cite[Section~2.2]{Friz2018}. In particular, we recall that, for $X \in D^p([0,T];\R^d)$ and $Y \in D^q([0,T];\cL(\R^d;\R^n))$, the Young integral
\begin{equation*}
  \int_s^t Y_r\dd X_r := \lim_{|\cP| \to 0} \sum_{[u,v] \in \cP} Y_u X_{u,v},
\end{equation*}
exists (in the classical mesh Riemann--Stieltjes sense) whenever $1/p + 1/q > 1$, and comes with the estimate
\begin{equation}\label{eq:Young integral estimate}
  \bigg|\int_s^t Y_r\dd X_r - Y_s X_{s,t}\bigg| \leq C_{p,q}\|Y\|_{q,[s,t)} \|X\|_{p,[s,t]},
\end{equation}
where the constant $C_{p,q}$ depends only on $p$ and $q$.

Let us remark that, in the presence of jumps, it is crucial to take left-point Riemann sums to define the Young integral since, for instance, mid- or right-point Riemann sums approximation lead in general to different limits. This is in contrast to Young integration for continuous paths. Moreover, we note that the Young integral itself $t \mapsto \int_0^t Y_r\dd X_r$ is a \cadlag path and its jump at time $t \in (0,T]$ is given by
\begin{equation*}
  \Delta\bigg(\int_0^\cdot Y_r\dd X_r \bigg)_t = Y_{t-}\Delta X_t.
\end{equation*}

\begin{remark}
  Despite our focus here on Young integration in the sense described above, it is actually necessary to instead define the integral in \eqref{eq:K is minimal} in the Lebesgue--Stieltjes sense. Suppose for instance that $A = 0$, $X = 0$, $L = 0$ and $y = 0$. Then, for any fixed $u \in (0,T]$, setting $Y_t = K_t = \1_{\{t \geq u\}}$ defines a solution $(Y,K)$ of \eqref{eq:reflected Young equation} such that \eqref{eq:K is minimal} holds in the Young (or equivalently Riemann--Stieltjes) sense, essentially because the left-endpoint always lies before the jump. Thus, Young integration does not correctly capture the minimality property for $K$ in the c{\`a}dl{\`a}g setting.
\end{remark}

The problem of proving existence and uniqueness results for reflected (stochastic) differential equations is known to be closely related to the so-called Skorokhod problem, as originally introduced by Skorokhod \cite{Skorohod1961}. Since our approach to the reflected Young differential equation~\eqref{eq:reflected Young equation} relies on the Skorokhod problem, we shall recall some properties of the Skorokhod problem in the next subsection and provide some basic estimates regarding $p$-variation semi-norms as groundwork for later purposes.

\subsection{Skorokhod problem and p-variation estimates}

Let $Y,L\in D([0,T];\R^n)$ be such that $Y_0 \geq L_0$. A solution to the Skorokhod problem associated with the path $Y$ and the lower barrier $L$, is a pair $(Z,K) \in D([0,T];\R^n) \times D([0,T];\R^n)$ such that
\begin{enumerate}[(a)]
  \item $Z_t = Y_t+K_t\geq L_t$ for $t\in [0,T]$,
  \item $K_0=0$ and $K=(K^1,\dots,K^n)$, where $K^i$ is non-decreasing function such that 
  \begin{equation*}
    \int_0^t(Z^i_s-L^i_s)\dd K^i_s=0, \qquad \text{for all} \quad t \in [0,T],
  \end{equation*}
  for every $i=1,\dots,n$, where the integral is understood in the Lebesgue--Stieltjes sense as before.
\end{enumerate}
It is well-known, that the Skorokhod problem has a unique solution~$(Z,K)$, see e.g.~\cite[Theorem~2.6 and Remark~2.7]{Burdzy2009}. Moreover, we introduce the associated Skorokhod map~$\cS$ by
\begin{equation*}
  \mathcal{S} \colon (Y,L) \to (Z,K)
\end{equation*}
where $(Z,K)$ is the solution to the Skorokhod problem given $(Y,L)$, and we set 
\begin{equation*}
  \mathcal{S}_1(Y,L) := Z\quad\text{and} \quad \mathcal{S}_2(Y,L) := K.
\end{equation*}
As the following result from \cite{Falkowski2015} shows, the Skorokhod map~$\mathcal{S}$ turns out to be a Lipschitz continuous map with respect to the $p$-variation distance.

\begin{theorem}[Theorem~2.2 in \cite{Falkowski2015}]\label{thm:Skorokhod Lipschitz est}
  Let $Y,L,\tilde{Y},\tilde{L}\in D([0,T];\R^n)$ and assume that $Y_0 \geq L_0$ and $\tilde{Y}_0 \geq \tilde{L}_0$. Let $(Z,K) = \mathcal{S}(Y,L)$ and $(\tilde{Z},\tilde{K}) = \mathcal{S}(\tilde{Y},\tilde{L})$ be the solutions of the corresponding Skorokhod problems. We have
  \begin{equation*}
    \|Z-\tilde{Z}\|_{p,[0,T]}  
    \leq C\Big(\|Y-\tilde{Y}\|_{p,[0,T]} + | Y_0-\tilde{Y}_0| + \|L-\tilde{L}\|_{p,[0,T]} + |L_0-\tilde{L}_0|\Big)
  \end{equation*}
  and 
  \begin{equation*}
    \|K-\tilde{K}\|_{p,[0,T]} 
    \leq C\Big(\|Y-\tilde{Y}\|_{p,[0,T]} + | Y_0-\tilde{Y}_0| + \|L-\tilde{L}\|_{p,[0,T]} + |L_0-\tilde{L}_0|\Big),
  \end{equation*}
  where the constant $C$ depends only on the dimension $n$.
\end{theorem}

By setting $\tY_t = Y_0$ and $\tL_t = L_0$ for all $t \in [0,T]$, we see that, under the assumptions of Theorem~\ref{thm:Skorokhod Lipschitz est}, we also have
\begin{equation}\label{eq:bound Skorokhod map}
  \|Z\|_{p,[0,T]} + \|K\|_{p,[0,T]}  \leq C\Big(\|Y\|_{p,[0,T]} + \|L\|_{p,[0,T]}\Big).
\end{equation}

\begin{remark}
  The Lipschitz continuity of the Skorokhod map with respect to the supremum norm is a classical result, see the works \cite{Dupuis1991,Dupuis1993} and \cite{Dupuis1999}, which treat the Skorokhod map on various (intricate) domains and with different types of reflections. Notably, it was observed in \cite{Ferrante2013} that the Skorokhod map~$\mathcal{S}$ is not Lipschitz continuous with respect to H{\"o}lder distances. Hence, it is essential to work with $p$-variation distances to treat reflected differential equations using the Skorokhod map, even when considering continuous driving signals~$A$ and~$X$.
\end{remark}

For later convenience, we collect here various useful estimates for $p$-variation norms.

\begin{lemma}\label{lem:p-variation estimates}
  If $1\leq q\leq p<\infty$, $r\in [1,\infty)$ and $X\in D([0,T];\R^n)$, then
  \begin{equation*}
    \|X\|_{p,[0,T]}\leq \|X\|_{q,[0,T]}, \quad \|X\|_{\infty,[0,T]}\leq |X_0|+\|X\|_{r,[0,T]}\quad \text{and}\quad \|X\|_{p,[0,T]}\leq \|X\|_{q,[0,T]}^{\frac{q}{p}} \|X\|_{r,[0,T]}^{1-\frac{q}{p}}.
  \end{equation*}
\end{lemma}

\begin{proof}
  The first inequality follows immediately from the corresponding result for classical $l^p$ spaces. The second inequality is straightforward to see by noting that 
  \begin{equation*}
    |X_t|\leq |X_0|+|X_{0,t}| \leq |X_0|+ \|X\|_{r,[0,T]}
  \end{equation*}
  for every $t \in [0,T]$. For the third inequality we observe that
  \begin{equation*}
    \|X\|_{p,[0,T]}^p \leq \|X\|_{q,[0,T]}^{q} \bigg( \sup_{s,t\in [0,T]}|X_{s,t}|^{p - q}\bigg)
    \leq \|X\|_{q,[0,T]}^{q} \|X\|_{r,[0,T]}^{p - q},
  \end{equation*}
  and the result follows.
\end{proof}

\begin{lemma}\label{lemma monotone paths p-var}
  There exists a constant $C_n$, depending only on the dimension $n$, such that
  \begin{equation}\label{eq:1var leq Cn pvar}
    \|X\|_{1,[0,T]} \leq C_n\|X\|_{p,[0,T]}
  \end{equation}
  for every monotone path $X \colon [0,T] \to \R^n$ (i.e.~any path $X$ such that each of its $n$ components $X^i \colon [0,T] \to \R$, $i = 1,\ldots,n$, is monotone) and every $p \geq 1$.

  Moreover, we may take $C_1 = 1$, so that for any one-dimensional monotone path $X$, the $p$-variation norm $\|X\|_{p,[0,T]}$ is independent of $p$.
\end{lemma}

\begin{proof}
  It is clear that
  \begin{equation}\label{eq:x0T leq pvar leq 1var}  
    |X_{0,T}| \leq \|X\|_{p,[0,T]} \leq \|X\|_{1,[0,T]}
  \end{equation}
  for any path $X$ and any $p \geq 1$. Suppose now that each of the components $X^i \colon [0,T] \to \R$, $i = 1,\ldots,n$, is monotone. Let us consider the $p$-variation of $X$ with distance in $\R^n$ measured using the $l^1$-norm rather than the usual Euclidean $l^2$-norm, so that $|x| = \sum_{i=1}^n |x^i|$. Since $X$ is monotone, it is then straightforward to see that $\|X\|_{1,[0,T]} = |X_{0,T}|$. Combining this with \eqref{eq:x0T leq pvar leq 1var}, we obtain $\|X\|_{1,[0,T]} = \|X\|_{p,[0,T]}$ for every $p \geq 1$.
 
  To change back to the usual Euclidean norm, we recall that norms on finite-dimensional spaces are equivalent, so that~\eqref{eq:1var leq Cn pvar} holds for a suitable constant $C_n$.
\end{proof}

\begin{lemma}\label{lem:p-variation have open interval}
  Let $X \in D^p([0,T];\R^n)$. 
  For any $0 \leq s < t \leq T$, we have
  \begin{equation*}
    \Big(\|X\|_{p,[s,t)}^p + |\Delta X_t|^p\Big)^{\frac{1}{p}} \leq \|X\|_{p,[s,t]} \leq \|X\|_{p,[s,t)} + |\Delta X_t|
  \end{equation*}
  where we recall that $\|X\|_{p,[s,t)}:= \sup_{u < t} \|X\|_{p,[s,u]}$.
\end{lemma}

\begin{proof}
  For the first inequality, note that
  \begin{equation*}
    \|X\|_{p,[s,u]}^p + \|X\|_{p,[u,t]}^p \leq \|X\|_{p,[s,t]}^p
  \end{equation*}
  for all $s < u < t$, and take the limit as $u \nearrow t$.

  For the second inequality, let $\cP = \{s = u_0 < u_1 < \cdots < u_n = t\}$ be a partition of the interval $[s,t]$. By Minkowski's inequality, we have
  \begin{align*}
    \bigg(\sum_{i=0}^{n-1} |X_{u_i,u_{i+1}}|^p\bigg)^{\frac{1}{p}} 
    &=\bigg(\sum_{i=0}^{n-2} |X_{u_i,u_{i+1}}|^p + |X_{u_{n-1},t-} + \Delta X_t|^p\bigg)^{\frac{1}{p}}\\
    & \leq \bigg(\sum_{i=0}^{n-2} |X_{u_i,u_{i+1}}|^p+|X_{u_{n-1},t-}|^p\bigg)^{\frac{1}{p}} + |\Delta X_t|,
  \end{align*}
  and we conclude by taking the supremum over all partitions $\cP$ of the interval $[s,t]$.
\end{proof}

\subsection{Existence and uniqueness result}

In this subsection we show the existence of a unique solution to the reflected Young differential equation~\eqref{eq:reflected Young equation}. We recall that we call $(Y,K)$ a solution to the reflected Young differential equation~\eqref{eq:reflected Young equation} driven by $A \in D^q([0,T];\R^d)$ and $X \in D^p([0,T];\R^n)$ with reflecting barrier $L \in D^p([0,T];\R^n)$ if $Y, K \in D^p([0,T];\R^n)$ satisfy~\eqref{eq:reflected Young equation}, and if for every $i = 1,\dots,n$,
\begin{enumerate}[(a)]
  \item $Y_t^i \geq L^i_t$ for all $t \in [0,T]$,
  \item $K^i\colon[0,T]\to \R$ is a non-decreasing function such that $K^i_0 = 0$, and
        \begin{equation*}
          \int_0^t (Y^i_s-L^i_s)\dd K^i_s = 0 \qquad \text{for every} \quad t \in [0,T].
        \end{equation*}
\end{enumerate}

\begin{theorem}\label{thm:existence uniqueness young case}
  Let $T>0$, $f \in C_b^2$, $q \in [1,2)$ and $p\in [q,\infty)$ such that $1/p+1/q>1$. Let $y \in \R^n$, $A\in D^q([0,T];\R^d)$, $X \in D^p([0,T];\R^n)$ and $L \in D^p([0,T];\R^n)$ such that $y \geq L_0$. There exists a unique solution $(Y,K) $ to the reflected Young differential equation~\eqref{eq:reflected Young equation}.
\end{theorem}

Before coming to the proof of Theorem~\ref{thm:existence uniqueness young case}, we first need the following stability result regarding Young integration.

\begin{lemma}\label{lem:Young contraction}
  Let $f \in C_b^2$, and let $q \in [1,2)$ and $p\geq q$ such that $1/p+1/q>1$. Let $A,\tilde{A} \in D^q([0,T];\R^d)$, $Y,\tY \in D^p([0,T];\R^n)$, and suppose there exists some $t \in (0,T]$ such that $\|\tY\|_{p,[0,t]}\leq 1$. Then
  \begin{align*}
    &\bigg\|\int_0^\cdot f(Y_r)\dd A_r - \int_0^\cdot f(\tY_r)\dd \tA_r\bigg\|_{p,[0,t]}\\
    &\quad\leq C_{p,q}\|f\|_{C^2_b}\Big(\Big(|Y_0 - \tY_0| + \|Y - \tY\|_{p,[0,t]}\Big)\|A\|_{q,[0,t]} + \|A - \tA\|_{q,[0,t]}\Big),
  \end{align*}
  where the constant $C_{p,q}$ depends only on $p$ and $q$.
\end{lemma}

\begin{proof}
  For any subinterval $[s,u] \subseteq [0,t]$, we have
  \begin{align*}
    \bigg|\int_s^u f(Y_r)\dd A_r - \int_s^u f(\tY_r)\dd \tA_r\bigg|
    &\leq \bigg|\int_s^u(f(Y_r) - f(\tY_r))\dd A_r\bigg| + \bigg|\int_s^u f(\tY_r)\dd (A - \tA)_r\bigg|\\
    &\lesssim |f(Y_s) - f(\tY_s)||A_{s,u}| + \|f(Y) - f(\tY)\|_{p,[s,u)}\|A\|_{q,[s,u]}\\
    &\qquad + |f(\tY_s)||A_{s,u} - \tA_{s,u}| + \|f(\tY)\|_{p,[s,u)}\|A - \tA\|_{q,[s,u]},
  \end{align*}
  where we applied~\eqref{eq:Young integral estimate} to obtain the last inequality. Hence, we deduce that
  \begin{align*}
    &\bigg\|\int_0^\cdot f(Y_r)\dd A_r - \int_0^\cdot f(\tY_r)\dd \tA_r\bigg\|_{p,[0,t]}\\
    &\lesssim \|f(Y) - f(\tY)\|_\infty\|A\|_{p,[0,t]} + \|f(Y) - f(\tY)\|_{p,[0,t]}\|A\|_{q,[0,t]}\\
    &\qquad + \|f(\tY)\|_\infty\|A - \tA\|_{p,[0,t]} + \|f(\tY)\|_{p,[0,t]}\|A - \tA\|_{q,[0,t]}\\
    &\lesssim \|f\|_{C^2_b}\Big(\Big(|Y_0 - \tY_0| + \|Y - \tY\|_{p,[0,t]}\Big)\Big(\|A\|_{q,[0,t]} + \|A\|_{p,[0,t]}\Big) + \|A - \tA\|_{q,[0,t]} + \|A - \tA\|_{p,[0,t]}\Big),
  \end{align*}
  where in the last line we used the fact that $\|f(\tY)\|_{p,[0,t]} \lesssim \|\tY\|_{p,[0,t]} \leq 1$, and the term $\|f(Y) - f(\tY)\|_{p,[0,t]}$ was bounded using \cite[Lemma~3.1]{Friz2018}. Since $p \geq q$, the first inequality in Lemma~\ref{lem:p-variation estimates} yields the assertion.
\end{proof}

We are now ready to establish the existence of a unique solution to the reflected Young differential equation~\eqref{eq:reflected Young equation}, the key ingredients being the Skorokhod map and Banach's fixed point theorem. Recall that the space $D^p([0,T];\R^n)$ is a Banach space with respect to the $p$-variation norm $|X_0| + \|X\|_{p,[0,T]}$ for $X \in D^p([0,T];\R^n)$ (see e.g.~\cite[Proposition~7.2]{Chistyakov1998}).

\begin{proof}[Proof of Theorem~\ref{thm:existence uniqueness young case}]
  \textit{Step~1: Local solution.} For $t \in (0,T]$ we define the map
  \begin{equation*}
    \cM_t \colon D^p([0,t];\R^n) \to D^p([0,t];\R^n)\quad  \text{by}\quad \cM_t(Y) := \cS_1\bigg(y + \int_0^\cdot f(Y_r)\dd A_r+X,\hspace{1pt} L\bigg).
  \end{equation*}
  That is, we have 
  \begin{equation*}
    \cM_t(Y) = y + \int_0^\cdot f(Y_r)\dd A_r + X + K, \quad 
    \text{where}\quad K = \cS_2 \bigg(y + \int_0^\cdot f(Y_r)\dd A_r+X,\hspace{1pt} L\bigg).
  \end{equation*}
  Note that a unique fixed point of the map $ \cM_t$, along with the corresponding process~$K$ obtained from the Skorokhod map~$\cS$, are together the unique solution to the reflected Young differential equation~\eqref{eq:reflected Young equation} over the time interval $[0,t]$. To show the existence of a unique fixed point, it is sufficient to verify that the map~$\cM_t$ satisfies the assumptions of Banach's fixed point theorem (\cite[Theorem~1.A]{Zeidler1986}) for some sufficiently small $t \in (0,T]$.
 
  \textit{Invariance.} We define the closed ball
  \begin{equation*}
    \cB_t := \big\{Y \in D^p([0,t];\R^n)\ \colon\ Y_0 = y,\ \|Y\|_{p,[0,t]} \leq 1\big\}.
  \end{equation*}
  Let $Y \in \cB_t$. By \eqref{eq:Young integral estimate}, for any subinterval $[s,u] \subseteq [0,t]$, we have
  \begin{equation*}
    \bigg|\int_s^u f(Y_r)\dd A_r+X_{s,u}\bigg| \lesssim |f(Y_s)A_{s,u}| + \|f(Y)\|_{p,[s,u)}\|A\|_{q,[s,u]}+|X_{s,u}|,
  \end{equation*}
  from which it follows that
  \begin{align*}
    \bigg\|\int_0^\cdot f(Y_r)\dd A_r + X \bigg\|_{p,[0,t]} 
    &\lesssim \|f(Y)\|_\infty\|A\|_{p,[0,t]} + \|f(Y)\|_{p,[0,t]}\|A\|_{q,[0,t]}
    +\|X\|_{p,[0,t]} \\
    &\lesssim \|f\|_{C^1_b}  \|A\|_{q,[0,t]}  +  \|X\|_{p,[0,t]},
  \end{align*}
  where we have used the first inequality in Lemma~\ref{lem:p-variation estimates}, and the fact that $\|f(Y)\|_{p,[0,t]} \lesssim \|Y\|_{p,[0,t]} \leq 1$. Hence, from the estimate~\eqref{eq:bound Skorokhod map} we get
  \begin{equation*}
    \|\cM_t(Y)\|_{p,[0,t]} \leq C_1\Big(\|f\|_{C^1_b}\|A\|_{q,[0,t]} +  \|X\|_{p,[0,t]} + \|L\|_{p,[0,t]}\Big)
  \end{equation*}
  for some constant $C_1$ depending only on $p, q$ and $n$. Since $A$, $X$ and $L$ are right-continuous, the functions $t\mapsto \|A\|_{q,[0,t]}$, $t\mapsto \|X\|_{p,[0,t]}$ and $t\mapsto \|L\|_{p,[0,t]}$ are non-decreasing and right-continuous, see \cite[Lemma~7.1]{Friz2018}. Hence, there exists $t_1 \in (0,T]$ sufficiently small such that
  \begin{equation*}
    C_1\big(\|f\|_{C^1_b}\|A\|_{q,[0,t_1]} +\|X\|_{p,[0,t_1]} +\|L\|_{p,[0,t_1]}\big) \leq 1
  \end{equation*}
  and it follows that for any $t \in (0,t_1]$, the closed ball $\cB_t$ is invariant under the map $\cM_t$.

  \textit{Contraction property.}
  Let $Y,\tY \in \cB_t$ for some $t \in (0,t_1]$. By Lemma~\ref{lem:Young contraction}, we have
  \begin{equation*}
    \bigg\|\int_0^\cdot f(Y_r)\dd A_r +X - \int_0^\cdot f(\tY_r)\dd A_r-X\bigg\|_{p,[0,t]} \lesssim \|f\|_{C^2_b}\|Y - \tY\|_{p,[0,t]}\|A\|_{q,[0,t]}.
  \end{equation*}
  By the first estimate in Theorem~\ref{thm:Skorokhod Lipschitz est}, we then have that
  \begin{equation*}
    \|\cM_t(Y) - \cM_t(\tY)\|_{p,[0,t]} \leq C_2\|f\|_{C^2_b}\|Y - \tY\|_{p,[0,t]}\|A\|_{q,[0,t]}
  \end{equation*}
  for some constant $C_2$ depending only on $p, q$ and $n$. Choosing $t_2 \in (0,t_1]$ sufficiently small so that $C_2\|f\|_{C^2_b}\|A\|_{q,[0,t_2]} \leq \frac{1}{2}$, it follows that, for any $t \in (0,t_2]$, the map $\cM_t$ is a contraction on $\cB_t$. Applying Banach fixed point theorem provides a unique~$Y\in \cB_t$ (together with a reflector term~$K\in D^1([0,t];\R^n$) satisfying the reflected Young differential equation~\eqref{eq:reflected Young equation} for any $t\in (0,t_2]$. Note that any solution~$\tilde{Y}\in D([0,T];\R^n)$ (and $\tilde K\in D^1([0,T];\R^n)$) to the reflected Young differential equation~\eqref{eq:reflected Young equation} belongs to the ball $\cB_t$ for sufficiently small $t\in (0,t_2]$, and thus $Y\equiv \tilde Y$ and $K\equiv\tilde K$ on $[0,t]$ for any $t\in [0,t_2]$.

  \textit{Step~2: Global solution.} Due to Step~1, we know that there exists a unique solution $(Y,K)$ to the Young differential equation~\eqref{eq:reflected Young equation} on every interval $[s,t)$ provided $\|A\|_{q,[s,t)}$, $\|X\|_{p,[s,t)}$ and $\|L\|_{p,[s,t)}$ are sufficiently small such that
  \begin{equation}\label{eq:condition local existence}
    C_1\big(\|f\|_{C^1_b}\|A\|_{q,[s,t)} +\|X\|_{p,[s,t)} +\|L\|_{p,[s,t)}\big) +  C_2\|f\|_{C^2_b}\|A\|_{q,[s,t)} \leq \frac{1}{2}.
  \end{equation}
  Note that the condition~\eqref{eq:condition local existence} is independent of the initial condition~$y$. By the right-continuity of $A$, $X$ and $L$, for every $\delta >0$ there exists a partition $\mathcal{P}=\{0=t_0<t_1<\dots <t_N=T\}$ of $[0,T]$ such that
  \begin{equation*}
    \|A\|_{q,[t_i,t_{i+1})} +\|X\|_{p,[t_{i},t_{i+1})} +\|L\|_{p,[t_i,t_{i+1})} \leq \delta 
  \end{equation*}
  for all $i = 0,\dots, N-1$. Now we choose $\delta>0$ such that the condition~\eqref{eq:condition local existence} holds for every $[s,t]\in\mathcal{P}$. Hence, we can iteratively obtain a solution $(Y^i,K^i)$ to the reflected Young differential equation~\eqref{eq:reflected Young equation} on each interval $[t_i,t_{i+1})$ with initial condition 
  \begin{equation*}
    Y_{t_i} = Y_{t_i-} + f(Y_{t_i-})\Delta A_{t_i} + \Delta X_{t_i} + \Delta K_{t_i},
  \end{equation*}
  with
  \begin{equation}\label{eq:Delta K Young}
    \Delta K_{t_i} = \big[L_{t_i} - Y_{t_i-} - f(Y_{t_i-})\Delta A_{t_i} - \Delta X_{t_i}\big]^+,
  \end{equation}
  where $[\hspace{1pt}\cdot\hspace{1pt}]^+$ denotes the positive part, in the sense of \eqref{eq:positive part}. The resulting paths $Y, K \in D^p([0,T];\R^n)$ thus provide a solution to the reflected Young differential equation~\eqref{eq:reflected Young equation} on $[0,T]$.
  
  The minimality of the reflector term~$K$ and the perservation of the local jump structure under Young integration (see \cite[Lemma~2.9]{Friz2018}) ensure that \eqref{eq:Delta K Young} is the only valid choice for the jump $\Delta K_{t_i}$. The uniqueness of $(Y,K)$ on $[0,T]$ follows from this and the local uniqueness established in Step~1.
\end{proof}

\subsection{Stability results}

One of the key advantages of a pathwise analysis of stochastic differential equations are pathwise stability results regarding the solution map associated to a differential equation, which maps the driving signals, in our case $A$ and $X$, to the solution~$Y$ of the differential equation. Accordingly, in this subsection we derive stability results for the reflected Young differential equation~\eqref{eq:reflected Young equation}. The first one is with respect to the $p$-variation distance.

\begin{proposition}\label{prop:stability}
  Let $f \in C_b^2$, $q \in [1,2)$ and $p\in [q,\infty)$ such that $1/p + 1/q > 1$. Let $(Y,K)$ and $(\tY,\tK)$ be the unique solutions of the reflected Young differential equation~\eqref{eq:reflected Young equation} given the data $y, \ty \in \R^n$, $A ,\tA \in D^q([0,T];\R^d)$, $X, \tX \in D^p([0,T];\R^n)$ and $L, \tL \in D^p([0,T];\R^n)$ respectively, where as usual $y \geq L_0$ and $\ty \geq \tL_0$. Suppose that the norms $\|A\|_{q,[0,T]}$, $\|\tA\|_{q,[0,T]}$, $\|X\|_{p,[0,T]}$, $\|\tX\|_{p,[0,T]}$, $\|L\|_{p,[0,T]}$ and $\|\tL\|_{p,[0,T]}$ are all bounded by a given constant $M > 0$. Then, 
  \begin{align*}
    &\|Y - \tY\|_{p,[0,T]} + \|K - \tK\|_{p,[0,T]}\\
    & \quad\leq C_{M,f}\bigg(|y - \ty| + \|A - \tA\|_{q,[0,T]} + \|X - \tX\|_{p,[0,T]} + |L_{0} - \tL_{0}| + \|L - \tL\|_{p,[0,T]}\bigg)
  \end{align*}
  for some constant $C_{M,f}$ depending on $M, \|f\|_{C^2_b}, p, q$ and $n$.
\end{proposition}

\begin{proof}
  \textit{Step~1: Local estimate for sufficiently small intervals.}
  We recall from the proof of Theorem~\ref{thm:existence uniqueness young case} that the unique solution of~\eqref{eq:reflected Young equation} satisfies $\|Y\|_{p,[s,t)} \leq 1$, whenever the interval $[s,t)$ is sufficiently small such that~\eqref{eq:condition local existence} holds for the data $(A,X,L)$. Thus, by Lemma~\ref{lem:Young contraction}, on any interval $[s,t)$ such that \eqref{eq:condition local existence} holds for both $(A,X,L)$ and $(\tA,\tX,\tL)$, we have that
  \begin{align*}
    &\bigg\|\int_0^\cdot f(Y_r)\dd A_r +X - \int_0^\cdot f(\tY_r)\dd \tA_r -\tX\bigg\|_{p,[s,t)}\\
    &\quad\lesssim \big(|Y_s - \tY_s| + \|Y - \tY\|_{p,[s,t)}\big)\|A\|_{q,[s,t)} + \|A - \tA\|_{q,[s,t)} + \|X - \tX\|_{p,[s,t)}.
  \end{align*}
  By the estimates in Theorem~\ref{thm:Skorokhod Lipschitz est}, we then have
  \begin{align*}
    &\|Y - \tY\|_{p,[s,t)}  + \|K - \tK\|_{p,[s,t)}\\
    &\quad\leq C \bigg(\big(|Y_s - \tY_s| + \|Y - \tY\|_{p,[s,t)}\big)\|A\|_{q,[s,t)} + \|A - \tA\|_{q,[s,t)}\\
   &\qquad \qquad + \|X - \tX\|_{p,[s,t)} + |Y_s - \tY_s| + \|L - \tL\|_{p,[s,t)} + |L_s - \tL_s|\bigg)
  \end{align*}
  for some constant~$C$ depending on $p, q$ and $n$. If we suppose that the interval $[s,t)$ is sufficiently small that
  \begin{equation}\label{eq:A suff small}
    C\|f\|_{C^2_b}\|A\|_{q,[s,t)} \leq \frac{1}{2}
  \end{equation}
  then, after rearranging, we obtain
  \begin{align}\label{eq:short time Young soln est}
    \begin{split}
    &\|Y - \tY\|_{p,[s,t)} + \|K - \tK\|_{p,[s,t)}\\
    &\quad \lesssim \|A - \tA\|_{q,[s,t)} + |Y_s - \tY_s|+ \|X - \tX\|_{p,[s,t)} + |L_s - \tL_s| + \|L - \tL\|_{p,[s,t)}.
    \end{split}    
  \end{align}

  \textit{Step~2: Estimating the ``big'' jumps.}
  We estimate
  \begin{align}\label{eq:est Delta K}
  \begin{split}
    &|\Delta K_{t} - \Delta \tK_{t}| \\
    &\quad= \big|\big[L_{t} - Y_{t-} - f(Y_{t-})\Delta A_{t} - \Delta X_{t}\big]^+ - \big[\tL_{t} - \tY_{t-} - f(\tY_{t-})\Delta \tA_{t_i} - \Delta \tX_{t}\big]^+\big|\\
    &\quad\lesssim |L_{t} - \tL_{t}| + |Y_{t-} - \tY_{t-}| + |\Delta A_{t} - \Delta \tA_{t}| + |\Delta X_{t} - \Delta \tX_{t}|,
  \end{split}    
  \end{align}
  where the multiplicative constant indicated by the symbol $\lesssim$ depends on $\|f\|_{C^1_b}$ and on the bound $M$. Moreover,
  \begin{align}
    |\Delta Y_{t} - \Delta \tY_{t}| &= |f(Y_{t-})\Delta A_{t} + \Delta X_{t} + \Delta K_{t} - f(\tY_{t-})\Delta \tA_{t} - \Delta \tX_{t} - \Delta \tK_{t}|\nonumber\\
    &\lesssim |Y_{t-} - \tY_{t-}| + |\Delta A_{t} - \Delta \tA_{t}| + |\Delta X_{t} - \Delta \tX_{t}| + |\Delta K_{t} - \Delta \tK_{t}|\label{eq:est Delta Y}\\
    &\lesssim |Y_{t-} - \tY_{t-}| + |\Delta A_{t} - \Delta \tA_{t}| + |\Delta X_{t} - \Delta \tX_{t}| + |L_{t} - \tL_{t}|.\nonumber
  \end{align}
  Combining~\eqref{eq:est Delta K} and~\eqref{eq:est Delta Y}, we obtain
  \begin{align}
    &|\Delta Y_{t} - \Delta \tY_{t}| + |\Delta K_{t} - \Delta \tK_{t}|\nonumber\\
    &\quad \lesssim |Y_{t-} - \tY_{t-}| + |\Delta A_{t} - \Delta \tA_{t}| + |\Delta X_{t} - \Delta \tX_{t}| + |L_{t} - \tL_{t}|\label{eq:est Delta Y and Delta K}\\
    &\quad \leq |Y_s - \tY_s| + \|Y - \tY\|_{p,[s,t)} + |\Delta A_{t} - \Delta \tA_{t}| + |\Delta X_{t} - \Delta \tX_{t}| + |L_{t} - \tL_{t}|.\nonumber
  \end{align}
  By the second inequality in Lemma~\ref{lem:p-variation have open interval} we have
  \begin{align*}  
    &\|Y - \tY\|_{p,[s,t]} + \|K - \tK\|_{p,[s,t]}\\
    &\quad \leq \|Y - \tY\|_{p,[s,t)} + \|K - \tK\|_{p,[s,t)} + |\Delta Y_{t} - \Delta \tY_{t}| + |\Delta K_{t} - \Delta \tK_{t}|.
  \end{align*}
  Combining the estimates~\eqref{eq:short time Young soln est} and~\eqref{eq:est Delta Y and Delta K} and substituting into the above, we obtain
  \begin{align*}
    &\|Y - \tY\|_{p,[s,t]} + \|K - \tK\|_{p,[s,t]}\\
    &\quad \lesssim |Y_s - \tY_s| + \|A - \tA\|_{q,[s,t)} + \|X - \tX\|_{p,[s,t)} + |L_s - \tL_s| + \|L - \tL\|_{p,[s,t)}\\ 
    &\qquad + |\Delta A_{t} - \Delta \tA_{t}| + |\Delta X_{t} - \Delta \tX_{t}| + |L_{t}  - \tL_{t}|.
  \end{align*}
  By the first inequality in Lemma~\ref{lem:p-variation have open interval}, we deduce that
  \begin{align}\label{eq:est Y and K on s,t}
    \begin{split}
    &\|Y - \tY\|_{p,[s,t]}^p + \|K - \tK\|_{p,[s,t]}^p\\
    &\quad \lesssim |Y_s - \tY_s|^p + \|A - \tA\|_{q,[s,t]}^p + \|X - \tX\|_{p,[s,t]}^p + |L_s - \tL_s|^p + \|L - \tL\|_{p,[s,t]}^p.
    \end{split}
  \end{align}

  \textit{Step~3: Global estimate.} So far we have shown that the estimate~\eqref{eq:est Y and K on s,t} holds for every pair of times $s < t$ such that the conditions~\eqref{eq:condition local existence} and~\eqref{eq:A suff small} hold.

  Since the functions $(s,t) \mapsto \|A\|_{q,[s,t]}^q$, $(s,t) \mapsto \|X\|_{p,[s,t]}^p$, $(s,t) \mapsto \|L\|_{p,[s,t]}^p$ (and similarly for $\tA, \tX, \tL$) are superadditive and bounded by $M^q \vee M^p$, there exists a partition $\cP = \{0 = t_0 < \dots < t_N = T\}$, where the number of intervals $N$ depends only on $M, \|f\|_{C^2_b}, p, q$ and $n$, such that~\eqref{eq:condition local existence} and~\eqref{eq:A suff small} hold on each interval $[t_i,t_{i+1})$ for $i = 0,\ldots,N-1$. Thus, for each $i = 0,\ldots,N-1$, we have
  \begin{align*}
    &\|Y - \tY\|_{p,[t_i,t_{i+1}]}^p + \|K - \tK\|_{p,[t_i,t_{i+1}]}^p\\
    &\quad \lesssim |Y_{t_i} - \tY_{t_i}|^p + \|A - \tA\|_{q,[t_i,t_{i+1}]}^p + \|X - \tX\|_{p,[t_i,t_{i+1}]}^p + |L_{t_i} - \tL_{t_i}|^p + \|L - \tL\|_{p,[t_i,t_{i+1}]}^p.
  \end{align*}
  Writing $|Y_{t_i} - \tY_{t_i}| \leq |Y_{t_{i-1}} - \tY_{t_{i-1}}| + \|Y - \tY\|_{[t_{i-1},t_{i}]}$ and similarly for $L - \tL$, and pasting the estimate~\eqref{eq:est Y and K on s,t} on different intervals together, we see that
  \begin{align*}
    &\|Y - \tY\|_{p,[t_i,t_{i+1}]}^p + \|K - \tK\|_{p,[t_i,t_{i+1}]}^p\\  
    &\quad \lesssim |Y_0 - \tY_0|^p + |L_0 - \tL_0|^p + \sum_{j=0}^i \Big(\|A - \tA\|_{q,[t_j,t_{j+1}]}^p + \|X - \tX\|_{p,[t_j,t_{j+1}]}^p + \|L - \tL\|_{p,[t_j,t_{j+1}]}^p\Big).
  \end{align*} 
  We recall the standard estimate
  \begin{equation*}
    \|Y - \tY\|_{p,[0,T]} \leq N^{\frac{p - 1}{p}} \bigg(\sum_{i=0}^{N-1} \|Y - \tY\|_{p,[t_{i},t_{i+1}]}^p\bigg)^{\frac{1}{p}}
  \end{equation*}
  which holds similarly for $K - \tK$. Putting this together, and recalling that the number of partitions $N$ depends only on $M, \|f\|_{C^2_b}, p, q$ and $n$, the desired result follows.
\end{proof}

In probability theory one often likes to work with a variety of different distances on the Skorokhod space $D([0,T];\R^n)$. Following~\cite[Section~5.1]{Friz2018}, we can immediately reformulate the stability result (Proposition~\ref{prop:stability}) in terms of a Skorokhod $J_1$ $p$-variation distance. To this end, we let $\Lambda$ be the set of all time-changes, that is, increasing bijective functions $\lambda \colon [0,T] \to [0,T]$, and write $\|\lambda\| := \sup_{t\in[0,T]}|\lambda(t)-t|$ for $\lambda \in \Lambda$.

We define two Skorokhod $J_1$ $p$-variation distances, namely
\begin{equation*}
  \sigma_{p,[0,T]}((Y,K),(\tY,\tK))
  := \inf_{\lambda \in \Lambda} \Big(\|\lambda\| \vee (\|Y \circ \lambda - \tY\|_{p,[0,T]} + \|K \circ \lambda - \tK\|_{p,[0,T]})\Big)
\end{equation*}
and
\begin{align*}
  &\hat{\sigma}_{p,q,[0,T]}((A,X,L),(\tA,\tX,\tL))\\  
  &\quad:= \inf_{\lambda \in \Lambda} \Big(\|\lambda\| \vee (\|A \circ \lambda - \tA\|_{q,[0,T]} + \|X \circ \lambda - \tX\|_{p,[0,T]} + \|L \circ \lambda - \tL\|_{p,[0,T]} )\Big).
\end{align*}

\begin{corollary}
  Let $f \in C_b^2$, $q \in [1,2)$ and $p \in [q, \infty)$ such that $1/p + 1/q > 1$. Let $y, \ty \in \R^n$, $A, \tA \in D^q([0,T];\R^d)$, $X, \tX, L, \tL \in D^p([0,T];\R^n)$ such that $y \geq L_0$ and $\ty \geq \tL_0$. Let $(Y,K)$ and $(\tY,\tK)$ be the unique solutions of the reflected Young differential equation~\eqref{eq:reflected Young equation} corresponding to the data $(y,A,X,L)$ and $(\ty,\tA,\tX,\tL)$, respectively. Suppose that the norms $\|A\|_{q,[0,T]}$, $\|\tA\|_{q,[0,T]}$, $\|X\|_{p,[0,T]}$, $\|\tX\|_{p,[0,T]}$, $\|L\|_{p,[0,T]}$ and $\|\tL\|_{p,[0,T]}$ are all bounded by a given constant $M > 0$. Then
  \begin{equation*}
    \sigma_{p,[0,T]} ((Y,K),(\tilde{Y},\tilde{K})) \leq C_{M,f}\Big(|y - \ty| + |L_0 - \tL_0| + \hat{\sigma}_{p,q,[0,T]} ((A,X,L),(\tA,\tX,\tL))\Big)
  \end{equation*}
  for some constant $C_{M,f}$ depending on $M, \|f\|_{C^2_b}, p, q$ and $n$.
\end{corollary}

\begin{proof}
  Let $\epsilon > 0$. By the definition of the Skorokhod distance, there exists a $\lambda \in \Lambda$ such that
  \begin{align*}
    &\|\lambda\| \vee \Big(\|A \circ \lambda - \tA\|_{q,[0,T]} + \|X \circ \lambda - \tX\|_{p,[0,T]} + \|L\circ \lambda - \tilde{L} \|_{p,[0,T]}\Big)\\
    &\quad< \hat{\sigma}_{p,q,[0,T]} ((A,X,L),(\tA,\tX,\tL)) + \epsilon.
  \end{align*}
  Since $p$-variation norms are invariant under time-changes, it is straightforward to observe that $(Y \circ \lambda, K \circ \lambda)$ is the unique solution of the reflected Young differential equation~\eqref{eq:reflected Young equation} with data $(y,A \circ \lambda,X \circ \lambda,L \circ \lambda)$. Hence, by Proposition~\ref{prop:stability}, we have that
  \begin{align*}
    &\sigma_{p,[0,T]} ((Y,K),(\tilde{Y},\tilde{K})) \leq \|\lambda\| + \|Y \circ \lambda - \tY\|_{p,[0,T]} + \|K \circ \lambda - \tK\|_{p,[0,T]}\\
    &\quad\lesssim \|\lambda\| + |y - \ty| + \|A \circ \lambda - \tA\|_{q,[0,T]} + \|X \circ \lambda - \tX\|_{p,[0,T]} + |L_0 - \tL_0| + \|L \circ \lambda - \tL\|_{p,[0,T]}\\ 
    &\quad\lesssim |y - \ty| + |L_0 - \tL_0| + \hat{\sigma}_{p,q,[0,T]} ((A,X,L),(\tA,\tX,\tL)) + \epsilon.
  \end{align*}
  Letting $\epsilon \to 0$, we obtain the result.
\end{proof}

\section{Reflected RDEs -- Existence}\label{sec:rough existence}

In order to develop a pathwise theory for reflected differential equations covering stochastic differential equations driven by, e.g.~L{\'e}vy processes or martingales, Young integration is in general not sufficient. To treat such processes one needs to significantly extend the theory of Young integration to be able to deal with paths of lower regularity. One such extension is given by the theory of rough paths initiated by Lyons~\cite{Lyons1998}. In the next subsection we recall the notion of integration with respect to c{\`a}dl{\`a}g rough paths, following the works of Friz--Shekhar~\cite{Friz2017} and Friz--Zhang~\cite{Friz2018}.

\subsection{C{\`a}dl{\`a}g rough paths}

Let $\Delta_T := \{(s,t) \in [0,T]^2\,:\, s \leq t\}$ be the standard $2$-simplex. For a two-parameter function $\mathbb{X} \colon \Delta_T \to \R^{d \times d}$ we define
\begin{equation*}
  \|\mathbb{X}\|_{\frac{p}{2},[s,t]} := \bigg(\sup_{\cP \subset [s,t]} \sum_{[u,v] \in \mathcal{P}} |\mathbb{X}_{u,v}|^{\frac{p}{2}} \bigg)^{\frac{2}{p}}
\end{equation*}
where the supremum is taken over all partitions of the interval $[s,t]$. We write $D^{\frac{p}{2}}(\Delta_T;\R^d)$ for the space of all functions $\mathbb{X} \colon \Delta_T \to \R^{d \times d}$ which satisfy $\|\bbX\|_{\frac{p}{2},[0,T]} < \infty$ and such that the maps $s \mapsto \bbX_{s,t}$ for fixed~$t$, and $t \mapsto \bbX_{s,t}$ for fixed~$s$, are both c{\`a}dl{\`a}g. Moreover, we set
\begin{equation*}
  \Delta \bbX_t := \bbX_{t-,t} = \lim_{s \to t,\, s < t} \bbX_{s,t} \qquad \text{for} \quad t \in (0,T].
\end{equation*}

\medskip

The fundamental definition of a c{\`a}dl{\`a}g rough path was first introduced in \cite[Definition~12]{Friz2017}. For $p\in [2,3)$ a pair $\X = (X, \mathbb{X})$ is called a \textit{c{\`a}dl{\`a}g $p$-rough path} over $\R^d$ if
\begin{enumerate}[(i)]
  \item $X \in D^p([0,T];\R^d)$ and $\bbX \in D^{\frac{p}{2}}(\Delta_T;\R^d)$,
  \item $\bbX_{s,t} - \bbX_{s,u} - \bbX_{u,t} = X_{s,u} \otimes X_{u,t}$ for $ 0 \leq s \leq u \leq t \leq T$.
\end{enumerate}
In component form, condition (ii) states that $\bbX^{ij}_{s,t} - \bbX^{ij}_{s,u} - \bbX^{ij}_{u,t} = X^i_{s,u} X^j_{u,t}$ for every $i$ and $j$.

We denote the space of c{\`a}dl{\`a}g $p$-rough paths by $\mathcal{D}^p([0,T];\R^d)$. On the space $\mathcal{D}^p([0,T];\R^d)$ we use the natural seminorm
\begin{equation*}
  \ver{\X}_{p,[s,t]} := \|X\|_{p,[s,t]} + \|\bbX\|_{\frac{p}{2},[s,t]}, \qquad \text{for} \quad (s,t) \in \Delta_T,
\end{equation*}
and distance
\begin{equation*}
  \|\bX;\tbX\|_{p,[s,t]} := \|X - \tX\|_{p,[s,t]} + \|\bbX - \tbbX\|_{\frac{p}{2},[s,t]}, \qquad \text{for} \quad (s,t) \in \Delta_T,
\end{equation*}
whenever $\bX = (X,\bbX), \tbX = (\tX,\tbbX) \in \mathcal{D}^p([0,T];\R^n)$.

Suppose that $\X = (X,\mathbb{X}) \in \mathcal{D}^{p}([0,T];\R^d)$ is a c{\`a}dl{\`a}g $p$-rough path for $p \in [2,3)$. A pair $(Y,Y')$ is called a \textit{controlled path} with respect to $X$ if
\begin{equation*}
  Y \in D^p([0,T];\R^n), \quad Y' \in D^p([0,T];\cL(\R^d;\R^n)) \quad \text{and} \quad R^Y \in D^{\frac{p}{2}}(\Delta_T;\R^n),
\end{equation*}
where $R^Y$ is defined by
\begin{equation*}
  R^Y_{s,t} = Y_{s,t} - Y'_{s} X_{s,t} \qquad \text{for} \quad (s,t) \in \Delta_T.
\end{equation*}
The space of controlled paths is denoted by $\mathcal{V}_{X}^{p}([0,T];\R^n)$, and $Y'$ is called Gubinelli derivative of $Y$ (with respect to $X$). For two controlled paths $(Y,Y')\in \mathcal{V}^p_X([0,T];\R^n)$ and $(\tilde{Y},\tilde{Y}')\in \mathcal{V}^p_{\tilde{X}}([0,T];\R^n)$ we introduce 
\begin{equation*}
  \|Y,Y'\|_{p,[s,t]}:=|Y_s|+|Y'_s|+\|Y'\|_{p,[s,t]}+\|R^Y\|_{\frac{p}{2},[s,t]}
\end{equation*}
and 
\begin{equation*}
  d_{X,\tilde{X},p,[s,t]}(Y,Y';\tilde{Y},\tilde{Y}') := \|Y'-\tilde{Y}' \|_{p,[s,t]} + \|R^Y - R^{\tilde{Y}}\|_{\frac{p}{2},[s,t]},
\end{equation*}
for $(s,t) \in \Delta_T$. The linear space $\mathcal{V}_X^p([0,T];\R^n)$ of controlled paths equipped with the norm $\|\hspace{0.5pt}\cdot ,\cdot\|_{p,[0,T]}$ is a compete metric space, cf.~\cite[Section~3.2]{Friz2018}.

\medskip

Given $p \in [2,3)$, $\X = (X,\mathbb{X}) \in \mathcal{D}^{p}([0,T];\R^d)$ and $(Y,Y') \in \mathcal{V}_{X}^{p}([0,T];\R^n)$, the rough path integral
\begin{equation}\label{eq:rough path integral}
  \int_s^t Y_r \dd \X_r :=  \lim_{|\mathcal{P}([s,t])| \to 0} \sum_{[u,v]\in\mathcal{P}([s,t])} Y_u X_{u,v} + Y'_u \mathbb{X}_{u,v}, \qquad (s,t) \in \Delta_T,
\end{equation}
exists (in the classical mesh Riemann--Stieltjes sense), and comes with the estimate
\begin{equation*}
  \bigg|\int_s^t Y_r \dd \X_r - Y_s X_{s,t} - Y'_s \mathbb{X}_{s,t}\bigg| \leq C\Big(\|R^Y\|_{\frac{p}{2},[s,t)} \|X\|_{p,[s,t]} + \|Y'\|_{p,[s,t)} \|\mathbb{X}\|_{\frac{p}{2},[s,t]}\Big)
\end{equation*}
for some constant $C$ depending only on $p$; see \cite[Proposition~2.6]{Friz2018}. As already mentioned for Young integration with respect to c{\`a}dl{\`a}g  paths, it is crucial to take left-point Riemann sums in the definition of the c{\`a}dl{\`a}g  rough path integral~\eqref{eq:rough path integral}. Moreover, let us remark that the rough path integral $t \mapsto \int_0^t Y_r\dd \bX_r$ is again a c{\`a}dl{\`a}g path and its jump at time $t \in (0,T]$ is given by
\begin{equation*}
  \Delta \bigg(\int_0^\cdot Y_r\dd \bX_r \bigg)_t = Y_{t-}\Delta X_t + Y'_{t-} \Delta \bbX_t,
\end{equation*}
see \cite[Lemma~2.9]{Friz2018}.

\begin{lemma}\label{lem:bound rough integral}
  Let $f \in C^3_b$. Let $\bX = (X,\bbX) \in \mathcal{D}^p([0,T];\R^d)$ be a c{\`a}dl{\`a}g $p$-rough path for some $p \in [2,3)$, and suppose that $(Y,Y') \in \mathcal{V}^p_X([0,T];\R^n)$ is a controlled path such that $|Y'_0| + \|Y'\|_{p,[0,T]} + \|R^Y\|_{\frac{p}{2},[0,T]} \leq M$ for some $M > 0$. Then, $(f(Y),\D f (Y) Y')$ is a controlled path, and
  \begin{equation*}
    \bigg\|\int_0^\cdot f(Y_r) \dd \bX_r\bigg\|_{p,[0,T]} \leq C_{M,f,p} (1 + \|X\|_{p,[0,T]}^2)\ver{\bX}_{p,[0,T]}.
  \end{equation*}
\end{lemma}

\begin{proof}
  By \cite[Lemma~3.5]{Friz2018}, we have $(f(Y),\D f (Y) Y') \in \mathcal{V}^p_X([0,T];\cL(\R^d;\R^n))$ for $f \in C^3_b$.
  
  Since
  \begin{equation*}
    \bigg|\int_s^t f(Y_r) \dd \bX_r\bigg| = |f(Y_s)X_{s,t} + R^{\int_0^\cdot f(Y_r) \dd \bX_r}_{s,t}| \lesssim |X_{s,t}| + \bigg|R^{\int_0^\cdot f(Y_r) \dd \bX_r}_{s,t}\bigg|,
  \end{equation*}
  it follows that
  \begin{align*}
    \bigg\|\int_0^\cdot f(Y_r) \dd \bX_r\bigg\|_{p,[0,T]} 
    &\lesssim \|X\|_{p,[0,T]} + \bigg\|R^{\int_0^\cdot f(Y_r) \dd \bX_r}\bigg\|_{\frac{p}{2},[0,T]}\\
    &\lesssim (1 + \|X\|_{p,[0,T]}^2)\ver{\bX}_{p,[0,T]},
  \end{align*}
  where in the last line we used \cite[Lemma~3.6]{Friz2018}.
\end{proof}

\begin{lemma}\label{lem:contraction rough integrals}
  Let $f \in C^3_b$. Let $\bX = (X,\bbX), \tbX = (\tX,\tbbX) \in \mathcal{D}^p([0,T];\R^d)$ be two \cadlag $p$-rough paths for some $p \in [2,3)$, and let $(Y,Y') \in \mathcal{V}^p_X([0,T];\R^n)$ and $(\tY,\tY') \in \cV^p_{\tX}([0,T];\R^n)$ be controlled paths. Suppose that $\ver{\bX}_{p,[0,T]} \leq M$, $\ver{\tbX}_{p,[0,T]} \leq M$,
  \begin{equation*}
    |Y'_0| + \|Y'\|_{p,[0,T]} + \|R^Y\|_{\frac{p}{2},[0,T]} \leq M
    \quad \text{and}\quad
    |\tY'_0| + \|\tY'\|_{p,[0,T]} + \|R^{\tY}\|_{\frac{p}{2},[0,T]} \leq M,
  \end{equation*}
  for some $M > 0$. Then
  \begin{align*}
    &\bigg\|\int_0^\cdot f(Y_r) \dd \bX_r - \int_0^\cdot f(\tY_r) \dd \tbX_r\bigg\|_{p,[0,T]}\\
    &\quad\leq C_{M,f,p}\Big(|Y_0 - \tY_0| + |Y'_0 - \tY'_0| + \|Y' - \tY'\|_{p,[0,T]} + \|R^Y - R^{\tY}\|_{\frac{p}{2},[0,T]} + \|\bX;\tbX\|_{p,[0,T]}\Big).
  \end{align*}
\end{lemma}

\begin{proof}
  Since
  \begin{align*}
    &\bigg|\int_s^t f(Y_r) \dd \bX_r - \int_s^t f(\tY_r) \dd \tbX_r\bigg|\\
    &\quad\leq |f(Y_s)||X_{s,t} - \tX_{s,t}| + |f(Y_s) - f(\tY_s)||\tX_{s,t}| + \bigg|R^{\int_0^\cdot f(Y_r) \dd \bX_r}_{s,t} - R^{\int_0^\cdot f(\tY_r) \dd \tbX_r}_{s,t}\bigg|,
  \end{align*}
  it follows that
  \begin{align*}
    &\bigg\|\int_0^\cdot f(Y_r) \dd \bX_r - \int_0^\cdot f(\tY_r) \dd \tbX_r\bigg\|_{p,[0,T]}\\
    &\quad\lesssim \|X - \tX\|_{p,[0,T]} + \|f(Y) - f(\tY)\|_{\infty,[0,T]}\|\tX\|_{p,[0,T]} + \|R^{\int_0^\cdot f(Y_r) \dd \bX_r} - R^{\int_0^\cdot f(\tY_r) \dd \tbX_r}\|_{\frac{p}{2},[0,T]}.
  \end{align*}
  Noting that $\|f(Y) - f(\tY)\|_{\infty,[0,T]} \lesssim |Y_0 - \tY_0| + \|f(Y) - f(\tY)\|_{p,[0,T]}$, and employing \cite[Lemma~3.7]{Friz2018}, we deduce the desired estimate.
\end{proof}

\subsection{Existence result for reflected RDEs}

The aim of this section is to establish existence of solutions to reflected differential equations driven by c{\`a}dl{\`a}g $p$-rough paths for $p\in [2,3)$. We consider a c{\`a}dl{\`a}g $p$-rough path $\X \in \mathcal{D}^{p}([0,T];\R^d)$ and a barrier $L \in D^{p}([0,T];\R^n)$. We seek a controlled path $(Y,Y')$ together with a  process $K$ satisfying the \textit{reflected rough differential equation} (reflected RDE) 
\begin{equation}\label{eq:reflected RDE}
  Y_{t}=y +\int_{0}^{t}f(Y_{s})\dd \X_{s} +K_{t},\qquad t\in[0,T],
\end{equation}
where the integral is defined in the sense of~\eqref{eq:rough path integral}, such that, for every $i=1,\dots,n$,
\begin{enumerate}[(a)]
  \item $Y_t^i \geq L^i_t$ for all $t \in [0,T]$,
  \item $K^i \colon [0,T] \to \R$ is a non-decreasing function such that $K^i_0 = 0$, and
  \begin{equation}\label{eq:minimality condition rough case}
    \int_0^t (Y^i_s-L^i_s)\dd K^i_s = 0, \qquad t \in [0,T],
  \end{equation}
  where the integral in~\eqref{eq:minimality condition rough case} is understood in the Lebesgue--Stieltjes sense.
\end{enumerate}
We call a triple $(Y,Y',K)$ a solution to the reflected RDE~\eqref{eq:reflected RDE} if $(Y,Y') \in \mathcal{V}_{X}^{p}([0,T];\R^n)$ and $K \in D^1([0,T];\R^n)$ satisfy \eqref{eq:reflected RDE} together with the conditions (a), (b). We remark that in general we cannot expect the Gubinelli derivative~$Y'$ to be uniquely determined by the RDE, but that the natural choice is known to be $Y' = f(Y)$. We refer to \cite[Sections~6.2 and~8.4]{Friz2014} for a more detailed discussion on the Gubinelli derivative and its uniqueness in the context of continuous rough paths.

\begin{remark}
  The natural generalization of the reflected Young differential equation~\eqref{eq:reflected Young equation} would arguably be the more general equation:
  \begin{equation*}
    Y_t = y + \int_0^t f_1(Y_s)\dd A_s+ \int_0^t f_2(Y_s)\dd \X_s + K_t, \qquad t \in [0,T],
  \end{equation*}
  subject to the conditions (a) and (b) above, where $A \in D^q([0,T];\R^d)$ and $\X \in \mathcal{D}^{p}([0,T];\R^d)$ for $p \in [2,3)$ and $1 \leq q \leq p$ such that $1/p + 1/q > 1$. However, since there is a canonical way to enhanced $A$ and $\X$ to a joint rough path (see \cite[Proposition~14 and~34]{Friz2017}), this equation can be readily reformulated into the form of~\eqref{eq:reflected RDE}.
\end{remark}

\begin{remark}
The rough integral $\int_{0}^{t}f(Y_{s})\dd \X_{s}$ appearing in the reflected RDE \eqref{eq:reflected RDE} is understood as the \textit{forward} rough integral as developed in \cite{Friz2018}, which corresponds to It{\^o} integration in a semimartingale setting. Alternatively, one can define the rough integral $\int_{0}^{t}f(Y_{s})\dd \X_{s}$ based on \textit{geometric} rough integration as introduced in \cite{Chevyrev2019}, which corresponds to Markus integration in a semimartingale setting. While Markus' geometric formulation would certainly lead to a natural formulation of reflected RDEs in a c{\`a}dl{\`a}g setting, the development of a Markus type theory for reflected RDEs requires a significantly different framework and methods compared to the present work; see \cite{Chevyrev2019} for the Markus type theory for (non-reflected) RDEs.
\end{remark}

The main result of this section is stated in the following theorem.

\begin{theorem}\label{thm:existence to reflected RDE}
  For $p\in [2,3)$ and $T>0$ let $\X=(X,\mathbb{X})\in \mathcal{D}^p([0,T];\R^d)$ be a c{\`a}dl{\`a}g $p$-rough path, $L\in D^p([0,T];\R^n)$ and $f \in C^3_b$. Then, for every $y \in \R^n$ with $y \geq L_0$ there exists a solution $(Y,Y', K)$ to the reflected RDE~\eqref{eq:reflected RDE} on $[0,T]$.
\end{theorem}

The proof of the existence result provided in Theorem~\ref{thm:existence to reflected RDE} is split into two parts. We first rely on Schauder's fixed point theorem to obtain a solution on sufficiently small intervals. In the second part we apply a pasting argument to construct a global solution, where we need to treat the finitely many ``big'' jumps of the driving signal by hand, similarly to the proof of Theorem~\ref{thm:existence uniqueness young case}.

\begin{proof}[Proof of Theorem~\ref{thm:existence to reflected RDE}]
  \textit{Step~1: Local solution.} Since the rough path $\X=(X,\mathbb{X})$ is c\`adl\`ag, the map $t\mapsto \ver{\X}_{p,[0,t]} $ is right-continuous with $\ver{\X}_{p,[0,0]}=0$. Hence, there exists a $t_1\in (0,T]$ such that $\ver{\X}_{p,[0,t]} \leq 1$ for every $t\in (0,t_1]$. Let us fix a $q\in (p,3)$. For $t\in (0,t_1]$ we introduce the solution map~$\mathcal{M}_t$ on the space of controlled paths by
  \begin{align}\label{eq:solution map reflected RDE}
  \begin{split}
    &\mathcal{M}_{t} \colon \mathcal{V}_{X}^{q}([0,t];\R^n)\to \mathcal{V}_{X}^{q}([0,t];\R^n),\\
    &\text{where }\ \mathcal{M}_{t}(Y,Y'):=\big(\mathcal{S}_1 (y+Z,L),f(Y)\big)
    \quad \text{with }\ Z_u:=\int_0^u f(Y_r)\dd \X_r, \quad u\in [0,t]. 
  \end{split}
  \end{align}
  First note that the map $\mathcal{M}_{t}$ is well-defined. Indeed, for $(Y,Y')\in \mathcal{V}_{X}^{q}([0,t];\R^n)$ we have $(Z,f(Y))\in \mathcal{V}_{X}^{q}([0,t];\R^n)$ cf.~\cite[Remark~2.8]{Friz2018}. That is, we have
  \begin{equation*}
    Z_{u,v}=f(Y_u)X_{u,v}+R^Z_{u,v}, \qquad \text{for all} \quad (u,v) \in \Delta_t,
  \end{equation*}
  where $R^Z\in D^{\frac{q}{2}}(\Delta_t;\R^n)$. Since 
  \begin{equation*}
    \mathcal{S}_1 (y+Z,L)_{u,v}= Z_{u,v} + K^Z_{u,v} = f(Y_u) X_{u,v} + R^Z_{u,v} + K^Z_{u,v}
  \end{equation*}
  where
  \begin{equation*}
    K^Z:=\mathcal{S}_2(y+Z,L) \in D^1([0,t];\R^n) \quad \text{and} \quad f(Y) \in D^q([0,t];\R^n),
  \end{equation*}
  we see that
  \begin{equation*}
    R^{\cS_1(y+Z,L)} := R^Z + K^Z \in D^{\frac{q}{2}}(\Delta_{t};\R^n),
  \end{equation*}
  so that $\mathcal{M}_{t}(Y,Y')\in \mathcal{V}_{X}^{q}([0,t];\R^n)$.
  
  We note that any fixed point of the map $\cM_t$, along with the corresponding process~$K$ obtained from the Skorokhod map~$\cS$, are together a solution to the reflected RDE~\eqref{eq:reflected RDE} over the time interval $[0,t]$. To show the existence of a fixed point, it is sufficient to verify that the map~$\cM_t$ satisfies the assumptions of Schauder's fixed point theorem (see e.g.~\cite[Theorem~2.A and Corollary~2.13]{Zeidler1986}) for sufficiently small $t\in (0,t_1]$. Recall that $\mathcal{V}_{X}^{q}([0,t];\R^n)$ equipped with the controlled path norm $\|\cdot,\cdot \|_{q,[0,t]}$ is a Banach space, cf.~\cite[Section~3.2]{Friz2018}.

  For $t \in (0,t_1]$ we define the ball 
  \begin{align*}
    \mathcal{B}_{t}:=\left \{(Y,Y')\in \mathcal{V}_{X}^{p}([0,t];\R^n)\,:
    \begin{array}{l}
      \, Y_0=y, \, Y'_0 =f(y), \,
      \|Y'\|_{p,[0,t]} \leq 1,\,\|R^Y \|_{\frac{p}{2},[0,t]} \leq 1,\\
      \|Y'\|_{p,[u,v]}\leq C_1 (\ver{\X}_{p,[u,v]}+\|L\|_{p,[u,v]}),\\
      \|R^Y \|_{\frac{p}{2},[u,v]}
      \leq C_2 (\ver{\X}_{p,[u,v]} + \|L\|_{p,[u,v]}),\\
      \text{for all }  (u,v) \in \Delta_t
    \end{array}
    \right\},
  \end{align*}
  for some suitable constants $C_1, C_2 \geq 1$ depending only on $f, p$ and $n$, which will be specified later. Let us remark that the ball $\mathcal{B}_t$ is a closed set with respect to $\|\cdot,\cdot \|_{q,[0,t]}$. Indeed, convergence in $\|\cdot,\cdot \|_{q,[0,t]}$ implies uniform convergence, and since every sequence in $\mathcal{B}_t$ has uniformly bounded $p$-variation, the uniform convergence ensures that its limit is again an element of $\mathcal{B}_t$ by the lower semi-continuity of $p$-variation norms, see e.g.~\cite[$(P7)$ in Section~2.2]{Chistyakov1998}.

  \textit{Compactness of the ball $\mathcal{B}_t$.} Due to the closedness of $\mathcal{B}_t$ it is sufficient to show that $\mathcal{B}_t$ is relatively compact. Further, due to the interpolation estimate for $q$-variation (see Lemma~\ref{lem:p-variation estimates}) and the lower semi-continuity of $p$-variation norms, it is sufficient to show that $\mathcal{B}_t$ is relatively compact with respect to the supremum norm. This follows since $\mathcal{B}_t$ is uniformly bounded in the supremum norm and equi-regulated by the definition of $\mathcal{B}_t$, see \cite[Proposition~1]{Cichon2018}.

  \textit{Invariance.} We shall show that there exists a $t_2 \in (0,t_1]$ such that, for every $t \in (0,t_2]$ we have $\mathcal{M}_{t} \colon \mathcal{B}_{t} \to \mathcal{B}_{t}$. Let $(Y,Y') \in \mathcal{B}_{t}$ be a controlled path with remainder $R^Y$ and let $(u,v) \in \Delta_t$. Recall that $\mathcal{M}_t(Y,Y')=(\mathcal{S}_1 (y+Z,L),f(Y))$ and thus the conditions $\mathcal{S}_1 (y+Z,L)_0=y$ and $f(Y_0) =f(y)$ are fulfilled. It remains to ensure the other conditions required in the definition of $\mathcal{B}_t$ are also fulfilled for sufficiently small~$t$. Using the fact that $\ver{\bX}_{p,[0,t]} \leq 1$, $\|Y'\|_{p,[0,t]} \leq 1$, $\|R^Y\|_{\frac{p}{2},[0,t]} \leq 1$, and $|Y'_u| \leq |Y'_0| + \|Y'\|_{p,[0,t]} \leq \|f\|_{\infty} + 1$, we have from~\cite[Lemma~3.6]{Friz2018} that
  \begin{equation*}
    \|f(Y)\|_{p,[u,v]} \lesssim \|X\|_{p,[u,v]} + \|R^Y\|_{\frac{p}{2},[u,v]} 
    \qquad \text{and} \qquad 
    \|R^{Z}\|_{\frac{p}{2},[u,v]} \lesssim \ver{\X}_{p,[u,v]}.
  \end{equation*}
  Since $\|R^Y\|_{\frac{p}{2},[u,v]} \leq C_2 (\ver{\X}_{p,[u,v]} + \|L\|_{p,[u,v]})$, we then have that
  \begin{equation}\label{eq:estimate of Gubinelli derivative}
    \|f(Y)\|_{p,[u,v]} \leq C_{f,p}(1 + C_2)(\ver{\X}_{p,[u,v]} + \|L\|_{p,[u,v]})
  \end{equation}
  for a constant $C_{f,p} \geq 1$ depending only on $f$ and $p$. Moreover, since $K^Z$ is non-decreasing, by Lemma~\ref{lemma monotone paths p-var}, we have that $\|K^Z\|_{\frac{p}{2},[u,v]} \lesssim \|K^Z\|_{p,[u,v]}$. Therefore, applying Theorem~\ref{thm:Skorokhod Lipschitz est} and Lemma~\ref{lem:bound rough integral} gives
  \begin{equation*}
    \|K^Z\|_{\frac{p}{2},[u,v]} \lesssim \|K^Z\|_{p,[u,v]}
    \lesssim \|Z\|_{p,[u,v]} + \|L\|_{p,[u,v]}
    \lesssim \ver{\bX}_{p,[u,v]} + \|L\|_{p,[u,v]}.
  \end{equation*}
  Since $R^{\cS_1(y+Z,L)} = R^Z + K^Z$, we get
  \begin{equation}\label{eq:estimate of remainder}
    \|R^{\cS_1(y+Z,L)}\|_{\frac{p}{2},[u,v]} 
    \leq \|R^Z\|_{\frac{p}{2},[u,v]} + \|K^Z\|_{\frac{p}{2},[u,v]} 
    \leq \hat{C}_{f,p} (\ver{\bX}_{p,[u,v]} + \|L\|_{p,[u,v]})
  \end{equation}
  for some constant $\hat{C}_{f,p} \geq 1$ depending only on $f, p$ and $n$.

  We now choose $C_1 = C_{f,p}(1 + \hat{C}_{f,p})$ and $C_2 = \hat{C}_{f,p}$. With these choices, the estimates~\eqref{eq:estimate of Gubinelli derivative} and \eqref{eq:estimate of remainder} become
  \begin{align*}
    \|f(Y)\|_{p,[u,v]} &\leq C_1 (\ver{\X}_{p,[u,v]} + \|L\|_{p,[u,v]}),\\
    \|R^{\cS_1(y+Z,L)}\|_{\frac{p}{2},[u,v]} &\leq C_2 (\ver{\bX}_{p,[u,v]} + \|L\|_{p,[u,v]}),
  \end{align*}
  which hold for all $(u,v) \in \Delta_{t_1}$.

  We then choose $t_2 \in (0,t_1]$ sufficiently small such that
  \begin{equation}\label{eq:condition local existence RDE}
    \ver{\bX}_{p,[0,t_2]} + \|L\|_{p,[0,t_2]} \leq \min \{C_1^{-1},C_2^{-1}\},
  \end{equation}
  so that in particular we have $\|f(Y)\|_{p,[0,t_2]} \leq 1$ and $\|R^{\cS_1(y+Z,L)}\|_{\frac{p}{2},[0,t_2]} \leq 1$. Thus, for every $t \in (0,t_2]$ we have shown that $\cM_t(Y,Y') \in \cB_t$ for all $(Y,Y') \in \cB_t$, that is, that $\mathcal{M}_{t} \colon \mathcal{B}_{t} \to \mathcal{B}_{t}$.

  \textit{Continuity.} We shall show that the map $\mathcal{M}_{t} \colon \mathcal{B}_{t} \to \mathcal{B}_{t}$ is $(1-2/q)$-H{\"o}lder continuous with respect to the controlled path norm $\|\cdot,\cdot\|_{q,[0,t]}$ for every $t \in (0,t_2]$. For $(Y,Y'), (\tilde{Y},\tilde{Y}') \in \mathcal{B}_{t}$ we write
  \begin{align*}
    Z_u := \int_0^u f(Y_r)\dd \X_r 
    \quad \text{and} \quad 
    \tilde{Z}_u := \int_0^u f(\tilde{Y}_r)\dd \X_r \quad \text{for} \quad u \in [0,t].
  \end{align*}
  We need to estimate
  \begin{align*}
    d_{X,X,q,[0,t]}(\cM_t(Y,Y');\cM_t(\tY,\tY')) 
    &= d_{X,X,q,[0,t]}(\cS_1(y + Z,L),f(Y);\cS_1(y + \tZ,L),f(\tY))\\  
    &= \|f(Y) - f(\tY)\|_{q,[0,t]} + \|R^{\cS_1(y + Z,L)} - R^{\cS_1(y + \tZ,L)}\|_{\frac{q}{2},[0,t]}.
  \end{align*}
  Since $R^{\cS_1(y + Z,L)} = R^Z + K^Z$ and $R^{\cS_1(y + \tZ,L)} = R^{\tZ} + K^{\tZ}$, we have
  \begin{equation*}
    \|R^{\mathcal{S}_1(y + Z,L)} - R^{\mathcal{S}_1(y + \tZ,L)}\|_{\frac{q}{2},[0,t]} \leq \|R^Z - R^{\tZ}\|_{\frac{q}{2},[0,t]} + \|K^Z - K^{\tZ}\|_{\frac{q}{2},[0,t]}.
  \end{equation*}
  The interpolation estimate in Lemma~\ref{lem:p-variation estimates} gives
  \begin{align*}
    \|K^Z - K^{\tZ}\|_{\frac{q}{2},[0,t]} \leq \|K^Z - K^{\tZ}\|_{1,[0,t]}^{\frac{2}{q}} \|K^Z - K^{\tZ}\|_{q,[0,t]}^{1 - \frac{2}{q}},
  \end{align*}
  and Theorem~\ref{thm:Skorokhod Lipschitz est} implies that
  \begin{equation*}
    \|K^Z - K^{\tZ}\|_{q,[0,t]} \lesssim \|Z - \tZ\|_{q,[0,t]}.
  \end{equation*}
  We recall from the inequalities \eqref{eq:estimate of Gubinelli derivative}, \eqref{eq:estimate of remainder} and \eqref{eq:condition local existence RDE}, that $\|f(Y)\|_{q,[0,t]} \leq \|f(Y)\|_{p,[0,t]} \leq 1$, $\|R^Z\|_{\frac{q}{2},[0,t]} \leq \|R^Z\|_{\frac{p}{2},[0,t]} \leq 1$ and, by Lemma~\ref{lemma monotone paths p-var}, that $\|K^Z\|_{1,[0,t]} \lesssim \|K^Z\|_{\frac{p}{2},[0,t]} \leq 1$. Combining the above estimates, we thus deduce that
  \begin{equation*}
    d_{X,X,q,[0,t]}(\cM_t(Y,Y');\cM_t(\tY,\tY'))
    \lesssim \|f(Y) - f(\tY)\|_{q,[0,t]}^{1 - \frac{2}{q}} + \|R^Z - R^{\tZ}\|_{\frac{q}{2},[0,t]}^{1 - \frac{2}{q}} + \|Z - \tZ\|_{q,[0,t]}^{1 - \frac{2}{q}}.
  \end{equation*}
  Using~\cite[Lemma~3.7]{Friz2018} and Lemma~\ref{lem:contraction rough integrals} we can bound the terms on the right-hand side, thus obtaining
  \begin{equation*}
    d_{X,X,q,[0,t]}(\cM_t(Y,Y');\cM_t(\tY,\tY')) \lesssim d_{X,X,q,[0,t]}(Y,Y';\tY,\tY')^{1 - \frac{2}{q}}.
  \end{equation*}

  \textit{Step~2: Global solution.} From Step~1, we know that there exists a solution $(Y,Y',K)$ to the reflected RDE~\eqref{eq:reflected RDE} on every interval $[s,t)$ such that $\ver{\bX}_{p,[s,t)}$ and $\|L\|_{p,[s,t)}$ are sufficiently small that they satisfy the bound in \eqref{eq:condition local existence RDE}. Note that the condition~\eqref{eq:condition local existence RDE} is independent of the initial condition $y$. By the right-continuity of $\X$ and $L$, the maps $t\mapsto \ver{\X}_{p,[s,t]} $ and $t\mapsto \|L\|_{p,[s,t]}$ are right-continuous with $\ver{\X}_{p,[s,s]} = \|L\|_{p,[s,s]} = 0$, for every $s\in [0,T]$. Hence, for every $\delta > 0$ there exists a partition $\cP = \{0 = t_0 < t_1 < \dots < t_N = T\}$ of the interval $[0,T]$ such that
  \begin{equation*}
    \ver{\bX}_{p,[t_{i},t_{i+1})} +\|L\|_{p,[t_i,t_{i+1})} \leq \delta
  \end{equation*}
  for all $i = 0,\dots,N - 1$. By choosing $\delta = \min \{C_1^{-1},C_2^{-1}\}$, we ensure that the condition~\eqref{eq:condition local existence RDE} holds for every $[s,t] \in \cP$. Hence, we can iteratively obtain a solution $(Y,f(Y),K)$ to the reflected RDE~\eqref{eq:reflected RDE} on each interval $[t_i,t_{i+1})$ with initial condition
  \begin{equation*}
    Y_{t_i} = Y_{t_i-} + f(Y_{t_i-})\Delta X_{t_i} + \D f(Y_{t_i-}) f(Y_{t_i-})\Delta \mathbb{X}_{t_i} + \Delta K_{t_i}, 
  \end{equation*}
  with
  \begin{equation}\label{eq:Delta K rough}
    \Delta K_{t_i} = [L_{t_i} - Y_{t_i-} - f(Y_{t_i-})\Delta X_{t_i} - \D f(Y_{t_i-}) f(Y_{t_i-})\Delta \mathbb{X}_{t_i}]^+,
  \end{equation}
  where $[\,\cdot\,]^+$ denotes the positive part. The minimality of the reflector term~$K$ and the perservation of the local jump structure under rough integration (see \cite[Lemma~2.9]{Friz2018}) ensure that~\eqref{eq:Delta K rough} is the only valid choice for the jump $\Delta K_{t_i}$. 

  Pasting the solutions on different intervals together, we obtain a solution $(Y,Y',K) = (Y,f(Y),K)$ to the reflected RDE~\eqref{eq:reflected RDE} on $[0,T]$.
\end{proof}

For Young and rough differential equations without reflection one can rely on Banach's fixed point theorem in order to show the existence of a unique solution. This strategy was still possible to implement in the case of reflected Young differential equations as we saw in Section~\ref{sec:young setting}. However, the situation for reflected rough differential equations is more intricate, and one is unable to rely on Banach's fixed point theorem.

\begin{remark}
  Recall that the solution map $\tilde{\mathcal{M}}_t$ associated to a (non-reflected) RDE is known to be locally Lipschitz continuous for sufficiently small $t$, that is
  \begin{align*}
    \tilde{\mathcal{M}}_t \colon \mathcal{V}_{X}^{p}([0,t];\R^n) \to \mathcal{V}_{X}^{p}([0,t];\R^n),
    \quad \text{via} \quad 
    \tilde{\mathcal{M}}_t(Y,Y') := \bigg(y + \int_0^{\cdot} f(Y_r)\dd \X_r,f(Y)\bigg),
  \end{align*}
  is locally Lipschitz continuous, see the proof of \cite[Theorem~3.8]{Friz2018}. Since the Skorokhod map~$\mathcal{S}$ is also Lipschitz continuous, one might expect the solution map~$\mathcal{M}_t$ associated to reflected RDEs, as defined in \eqref{eq:solution map reflected RDE}, to be locally Lipschitz continuous as well. However, this seems not to be the case, essentially because the controlled path space $\mathcal{V}_{X}^{p}([0,t];\R^n)$ is equipped with a stronger norm than $p$-variation. Indeed, one needs to consider
  \begin{equation*}
    \mathcal{V}_{X}^{p}([0,T];\R^n) \subset D^p([0,T];\R^n) \otimes D^{\frac{p}{2}}(\Delta_T;\R^{d}).
  \end{equation*}
  This makes a significant difference when extending the Skorokhod map from the $p$-variation space to the space of controlled paths. While the map
  \begin{equation*}
    \tilde{S} \colon \mathcal{V}_{X}^{p}([0,T];\R^n) \to \mathcal{V}_{X}^{p}([0,T];\R^n) \quad \text{via} \quad (Y,Y') \mapsto (Y + K,Y')
  \end{equation*}
  is Lipschitz continuous with respect to the distance $\|\cdot\|_{p,[0,T]} + \|\cdot\|_{p,[0,T]}$ (taking $ (Y + K,Y')$ as input), the extension $\tilde{S}$ is only locally H{\"o}lder continuous with respect to distance $d_{X,X,p,[0,T]}$, as shown by the interpolation argument used in the proof of Theorem~\ref{thm:existence to reflected RDE}. To improve the H{\"o}lder continuity of $\tilde{S}$ to (local) Lipschitz continuity with respect to the distance $d_{X,X,p,[0,T]}$ is, unfortunately, impossible; see \cite[Section~3.1]{Deya2019} for a discussion on this in the case of continuous driving signals.
\end{remark}

\section{Reflected RDEs -- Uniqueness in one-dimension}\label{sec:rough uniqueness}

For multidimensional reflected differential equations driven by $p$-rough paths with $p > 2$, it is known that uniqueness of solutions does not hold in general. Indeed, Gassiat~\cite{Gassiat2021} provides a linear rough differential equation in $n = 2$ dimensions reflected at $0$ which possesses infinitely many solutions. However, for one-dimensional reflected RDEs (i.e.~the solution $Y$ of the RDE is real-valued) uniqueness does hold for reflected differential equations driven by continuous rough paths, as proven by Deya et al.~\cite{Deya2019}, see also \cite{Richard2020}. The next theorem provides a uniqueness result for reflected one-dimensional RDEs driven by c{\`a}dl{\`a}g $p$-rough paths, i.e.~for the case when $n = 1$.

\begin{theorem}\label{thm:uniqueness RDEs}
  For $p \in [2,3)$ let $\bX = (X,\bbX) \in \mathcal{D}^p([0,T];\R^d)$ be a c{\`a}dl{\`a}g $p$-rough path, $L \in D^p([0,T];\R)$ and $f \in C^3_b$ with $n = 1$. Then, for every $y \in \R$ with $y \geq L_0$, there exists at most one solution $(Y,Y',K)$ with $Y' = f(Y)$ to the one-dimensional reflected RDE~\eqref{eq:reflected RDE}.
\end{theorem}

\begin{proof}
   Let $(Y,Y',K) = (Y,f(Y),K)$ and $(\tY,\tY',\tK) = (\tY,f(\tY),\tK)$ be two solutions of the reflected RDE~\eqref{eq:reflected RDE} given $\bX, L$ and $y$. Note that if $Y$ and $\tY$ are identical then $K$ and $\tK$ are also identical, as the corresponding Skorokhod problem has a unique solution. We assume for a contradiction that
   \begin{equation*}
     Y_a \neq \tY_a
   \end{equation*}
   for some $a \in (0,T]$.

  \textit{Step~1.}
  Let $u$ be the last time before $a$ that the two solutions $Y$ and $\tY$ were equal, i.e.
  \begin{equation}\label{eq:defn time u}
    u := \sup \big\{s \in [0,a) : Y_s = \tY_s\big\}.
  \end{equation}
  We claim that
  \begin{equation}\label{eq:Y eq tY at u}
    Y_u = \tY_u.
  \end{equation}
  To see this we first note that, by the definition of $u$, there exists a sequence of times $(r_k)_{k \geq 1}$ such that $Y_{r_k} = \tY_{r_k}$ for all $k$, and $r_k \nearrow u$ as $k \to \infty$. If $r_k = u$ for any $k$ then we are done, so we may instead assume that $r_k < u$ for all $k$. We observe that
  \begin{equation}\label{eq:Y jump at u}
    \Delta Y_{u} = f(Y_{u-})\Delta X_{u} + \D f(Y_{u-})f(Y_{u-})\Delta \bbX_{u} + \Delta K_{u}
  \end{equation}
  and similarly for $\Delta \tY_{u}$. Since $Y_{r_k} = \tY_{r_k}$ for all $k$, and $r_k \nearrow u$ as $k \to \infty$, we have that $Y_{u-} = \tY_{u-}$. Since $\Delta K_{u}$ is uniquely determined by $Y_{u-}$, $\Delta X_{u}$, $\Delta \bbX_{u}$ and $L_u$ by the relation
  \begin{equation*}
    \Delta K_{u} = [L_u - Y_{u-} - f(Y_{u-})\Delta X_{u} - \D f(Y_{u-})f(Y_{u-})\Delta \bbX_{u}]^+,
  \end{equation*}
  we see that $\Delta K_{u} = \Delta \tK_{u}$, and it then follows from~\eqref{eq:Y jump at u} that~\eqref{eq:Y eq tY at u} does indeed hold.

  Purely for notational simplicity, we shall henceforth assume without loss of generality that $u = 0$, i.e.~that the two solutions separate immediately after time $0$. Indeed, it then follows from \eqref{eq:defn time u} that
  \begin{equation}\label{eq:Ys neq tYs}
    Y_s \neq \tY_s \quad \text{for all} \quad s \in (0,a].
  \end{equation}
 
  \textit{Step~2.}
  We split the remainder of the proof into two cases. Namely, either there exists a time $l \in (0,a]$ such that the function
  \begin{equation*}
    s \, \mapsto \, K_s - \tK_s
  \end{equation*}
  is monotone on the interval $[0,l]$, or there does not.

  Let us first assume that there does not exist such a time $l$. It then follows that there exists a strictly decreasing sequence of times $(t_j)_{j \geq 1}$ with $t_j \in (0,a]$ and $t_j \searrow 0$ as $j \to \infty$, such that
  \begin{align*}
    K_{t_{2k},t_{2k-1}} - \tK_{t_{2k},t_{2k-1}} > 0
    \quad\text{and}\quad  
    K_{t_{2k+1},t_{2k}} - \tK_{t_{2k+1},t_{2k}} < 0,
  \end{align*}
  for every $k \geq 1$. Since $K$ and $\tK$ are both non-decreasing, this implies in particular that $K_{t_{2k},t_{2k-1}} > 0$ and $\tK_{t_{2k+1},t_{2k}} > 0$ for every $k \geq 1$.

  Since, by definition, the reflector $K$ only increases when $Y$ hits the barrier $L$, it follows that there exists another strictly decreasing sequence of times $(r_j)_{j \geq 1}$ with $r_j \in (t_{j+1},t_j]$ for every $j \geq 1$, such that
  \begin{equation*}
    Y_{r_{2k-1}} = L_{r_{2k-1}}
    \quad \text{and} \quad
    \tY_{r_{2k}} = L_{r_{2k}} \quad \text{for all} \quad k \geq 1. 
  \end{equation*}
  As $Y_{r_{2k-1}} = L_{r_{2k-1}} \leq \tY_{r_{2k-1}}$ and $Y_{r_{2k-1}} \neq \tY_{r_{2k-1}}$ (by \eqref{eq:Ys neq tYs}), and similarly at time $r_{2k}$, we see that
  \begin{equation*}
    Y_{r_{2k-1}} < \tY_{r_{2k-1}}
    \quad \text{and} \quad
    Y_{r_{2k}} > \tY_{r_{2k}} \quad \text{for all} \quad k \geq 1.
  \end{equation*}
  If the solutions $Y$ and $\tY$ were continuous, then it would follow immediately from the intermediate value theorem that there must exist a positive time (and actually infinitely many such times) $s \in (0,a]$ such that $Y_s = \tY_s$, contradicting \eqref{eq:Ys neq tYs}. Since our paths are only assumed to be c{\`a}dl{\`a}g, we must argue differently, as, at least in principle, the solutions may ``jump over each other'' infinitely many times.

  \textit{Step~3.}
  For each $k \geq 1$, we let 
  \begin{equation*}
    s_k := \inf \big\{t > r_{2k} : Y_t < \tY_t\big\},
  \end{equation*}
  which defines another strictly decreasing sequence of times $(s_k)_{k \geq 1}$ such that $s_k \searrow 0$ as $k \to \infty$. By right-continuity, we have that $Y_{s_k} \leq \tY_{s_k}$ which, by \eqref{eq:Ys neq tYs}, implies that
  \begin{equation*}
    Y_{s_k} < \tY_{s_k} \quad \text{for all} \quad k \geq 1.
  \end{equation*}
  It is clear that $Y_{s_k-} \geq \tY_{s_k-}$, but if $Y_{s_k-} = \tY_{s_k-}$ then a very similar argument to the one in Step 1 above would imply that $Y_{s_k} = \tY_{s_k}$, which would contradict~\eqref{eq:Ys neq tYs}. Thus, we must have that
  \begin{equation}\label{eq:Yskm ge tYskm}
    Y_{s_k-} > \tY_{s_k-} \quad \text{for all} \quad k \geq 1.
  \end{equation}

  Since $\tY_{s_k} > Y_{s_k} \geq L_{s_k}$, the minimality of the reflector $\tK$ implies that $\Delta \tK_{s_k} = 0$. We thus have that
  \begin{align*}
    0 > Y_{s_k} - \tY_{s_k} &= Y_{s_k-} - \tY_{s_k-} + \big(f(Y_{s_k-}) - f(\tY_{s_k-})\big)\Delta X_{s_k}\\
    &\quad + \big(\D f(Y_{s_k-})f(Y_{s_k-}) - \D f(\tY_{s_k-})f(\tY_{s_k-})\big)\Delta \bbX_{s_k} + \Delta K_{s_k}.
  \end{align*}
  Rearranging and using the fact that $K$ is non-decreasing, we obtain
  \begin{align*}
    0 &< Y_{s_k-} - \tY_{s_k-}\\  
    &< -\big(f(Y_{s_k-}) - f(\tY_{s_k-})\big)\Delta X_{s_k} - \big(\D f(Y_{s_k-})f(Y_{s_k-}) - \D f(\tY_{s_k-})f(\tY_{s_k-})\big)\Delta \bbX_{s_k}.
  \end{align*}
  As $f \in C^3_b$, we deduce the existence of a constant $C > 0$, depending only on $\|f\|_{C^2_b}$, such that
  \begin{equation*}
    |Y_{s_k-} - \tY_{s_k-}| \leq C|Y_{s_k-} - \tY_{s_k-}|\Big(|\Delta X_{s_k}| + |\Delta \bbX_{s_k}|\Big).
  \end{equation*}
  Since $Y_{s_k-} - \tY_{s_k-} \neq 0$ by~\eqref{eq:Yskm ge tYskm}, we deduce that
  \begin{equation*}
    |\Delta X_{s_k}| + |\Delta \bbX_{s_k}| \geq C^{-1} \quad \text{for every} \quad k \geq 1,
  \end{equation*}
  from which we conclude that  
  \begin{equation*}
    \|X\|_{p,[0,a]}^p + \|\bbX\|_{\frac{p}{2},[0,a]}^{\frac{p}{2}} \geq \sum_{k=1}^\infty |\Delta X_{s_k}|^p + |\Delta \bbX_{s_k}|^{\frac{p}{2}} = \infty,
  \end{equation*}
  contradicting the assumption that $\bX = (X,\bbX)$ is a $p$-rough path.

  \textit{Step~4.}
  Recall that in Step~2 we split the proof into two cases. We now proceed to the second case. Namely, we suppose that there exists a time $l \in (0,a]$ such that the function $s \mapsto K_s - \tK_s$ is monotone on the interval $[0,l]$. In particular, it follows from Lemma~\ref{lemma monotone paths p-var} that
  \begin{equation}\label{eq:K m tK eq p vars}
    \|K - \tK\|_{\frac{p}{2},[0,t]} = \|K - \tK\|_{p,[0,t]} \quad \text{for all} \quad t \in (0,l].
  \end{equation}
  Using~\eqref{eq:K m tK eq p vars}, Theorem~\ref{thm:Skorokhod Lipschitz est}, and an elementary estimate for controlled rough paths, we have that
  \begin{align*}
    \|K - \tK\|_{\frac{p}{2},[0,t]} &= \|K - \tK\|_{p,[0,t]}\\
    &\lesssim \bigg\|\int_0^\cdot f(Y_r)\dd \bX_r - \int_0^\cdot f(\tY_r)\dd \bX_r\bigg\|_{p,[0,t]}\\
    &\leq \|f(Y) - f(\tY)\|_{p,[0,t]}\|X\|_{p,[0,t]} + \left\|R^{\int_0^\cdot f(Y_r)\dd \bX_r} - R^{\int_0^\cdot f(\tY_r)\dd \bX_r}\right\|_{\frac{p}{2},[0,t]}.
  \end{align*}
  Let $\delta \geq 1$. As is clear from the structure of controlled rough paths, we have
  \begin{align*}
    &\|Y' - \tY'\|_{p,[0,t]} + \delta\|R^Y - R^{\tY}\|_{\frac{p}{2},[0,t]}\\
    &\quad\leq \|f(Y) - f(\tY)\|_{p,[0,t]} + \delta\left\|R^{\int_0^\cdot f(Y_r)\dd \bX_r} - R^{\int_0^\cdot f(\tY_r)\dd \bX_r}\right\|_{\frac{p}{2},[0,t]} + \delta\|K - \tK\|_{\frac{p}{2},[0,t]}\\
    &\quad\lesssim \|f(Y) - f(\tY)\|_{p,[0,t]} + \delta\left\|R^{\int_0^\cdot f(Y_r)\dd \bX_r} - R^{\int_0^\cdot f(\tY_r)\dd \bX_r}\right\|_{\frac{p}{2},[0,t]} + \delta\|Y' - \tY'\|_{p,[0,t]}\|X\|_{p,[0,t]}.
  \end{align*}
  Applying \cite[Lemma~3.7]{Friz2018}, we obtain
  \begin{align*}
    &\|Y' - \tY'\|_{p,[0,t]} + \delta\|R^Y - R^{\tY}\|_{\frac{p}{2},[0,t]}\\
    &\quad\leq C\Big(\|R^Y - R^{\tY}\|_{\frac{p}{2},[0,t]} + (1 + \delta)\Big(\|Y' - \tY'\|_{p,[0,t]} + \|R^Y - R^{\tY}\|_{\frac{p}{2},[0,t]}\Big)\ver{\bX}_{p,[0,t]}\Big)
  \end{align*}
  for some constant $C > 0$, independent of both $\delta \geq 1$ and $t \in (0,l]$. Let us now choose $\delta = 1 + C$. Then, we get
  \begin{align*}
    \|Y' - \tY'\|_{p,[0,t]} + \|R^Y - R^{\tY}\|_{\frac{p}{2},[0,t]}
    \leq C(2 + C)\Big(\|Y' - \tY'\|_{p,[0,t]} + \|R^Y - R^{\tY}\|_{\frac{p}{2},[0,t]}\Big)\ver{\bX}_{p,[0,t]}.
  \end{align*}
  Since the rough path $\bX = (X,\bbX)$ is c{\`a}dl{\`a}g, the function $t \mapsto \ver{\bX}_{p,[0,t]}$ is itself right-continuous (see \cite[Lemma~7.1]{Friz2018}), so we may choose $t \in (0,l]$ sufficiently small such that $C(2 + C)\ver{\bX}_{p,[0,t]} \leq \frac{1}{2}$. It then follows from the above that $\|Y' - \tY'\|_{p,[0,t]} = 0$ and $\|R^Y - R^{\tY}\|_{\frac{p}{2},[0,t]} = 0$, and hence that $\|Y - \tY\|_{p,[0,t]} = 0$. Thus, $Y = \tY$ on $[0,t]$, contradicting~\eqref{eq:Ys neq tYs}.
\end{proof}

\begin{remark}
  In the proof of Theorem~\ref{thm:uniqueness RDEs}, the assumption that the solution to the reflected RDE~\eqref{eq:reflected RDE} is one-dimensional is only crucial in Steps~2 and~3. In particular, the estimates in Step~4 may be reproduced without any additional difficulty in the multidimensional case. Thus, even in the multidimensional case, if non-uniqueness does occur, at time~$u$ say, then there does not exist an $l > 0$ such that the function $s \mapsto K_s - \tK_s$ is monotone on the interval $[u,u + l]$. It then follows, as we argued in Step~2, that uniqueness can only be lost directly after hitting the barrier, and that all solutions must hit the barrier infinitely many times immediately after uniqueness is lost. Indeed, this is precisely what happens in the counterexample of Gassiat, cf.~the proof of \cite[Theorem~2.1]{Gassiat2021}.
\end{remark}

While one cannot expect to obtain uniqueness for general multidimensional reflected RDEs, equations with specific vector fields can still be treated with the arguments developed in the proof of Theorem~\ref{thm:uniqueness RDEs}. To this end, we introduce following class of vector fields.

\begin{definition}
  We say that a map $f$ belongs to the class $\mathbb{L}^3_b$, if $f \in C^3_b(\R^n;\cL(\R^d;\R^n))$ and is such that each of its $n$ components is given by a function $f_i$, i.e.
  \begin{equation*}
    [f(y)(x)]_i = f_i(y;x) \qquad \text{for each} \quad i = 1,\ldots,n,
  \end{equation*}
  where, for each $i = 1,\ldots,n$, the map $f_i \colon \R^n \times \R^d \to \R$ only depends on its first $i$ arguments, that is,
  \begin{equation*}
    f_i(y_1,\dots,y_n;x) = f_i(y_1,\dots,y_i,\ty_{i+1},\ldots,\ty_n;x) \quad \text{for all} \quad y, \ty \in \R^n, x \in \R^d.
  \end{equation*}
\end{definition}

The structure of the vector fields in $\mathbb{L}^3_b$ allows one to recover uniqueness by successively applying the arguments of the proof of Theorem~\ref{thm:uniqueness RDEs} to each of the $n$ components of the equation in turn. We thus immediately obtain the following corollary.

\begin{corollary}
  For $p \in [2,3)$, let $\X=(X,\mathbb{X})\in \mathcal{D}^p([0,T];\R^d)$ be a c{\`a}dl{\`a}g $p$-rough path, $L\in D^p([0,T];\R^n)$ and $y \in \R^n$ such that $y \geq L_0$. If $f \in \mathbb{L}^3_b$, then there exists at most one solution $(Y,Y',K)$ with $Y' = f(Y)$ to the reflected RDE~\eqref{eq:reflected RDE}.
\end{corollary}

\begin{remark}\label{rem:Deya approach}
  If the driving signal $\bX$ and barrier $L$ are continuous, then one can also prove uniqueness for the one-dimensional reflected RDE~\eqref{eq:reflected RDE} via the rough Gr{\"o}nwall lemma of \cite{Deya2019b}, see \cite{Deya2019} and \cite{Richard2020}. This strategy crucially relies on the uniqueness argument of the sewing lemma (cf.~\cite[Lemma~1]{Deya2019}), which is in turn related to the existence of a suitably regular (i.e.~continuous) control function. However, in the presence of jumps it is not so straightforward to find such a regular control function. This approach thus does not seem feasible for the general c{\`a}dl{\`a}g setting considered here.

  More precisely, assuming $L=0$, in \cite{Deya2019} the authors applied a rough It{\^o} formula (see e.g.~\cite[Section~7.5]{Friz2014}) to $h(Y^1_t,Y^2_t) - h(Y^1_s,Y^2_s)$, where $Y^1$ and $Y^2$ are solutions to \eqref{eq:reflected RDE} and $h$ is a $C^3$-function which approximates the function $(y^1,y^2) \mapsto |y^1 - y^2|$. If $\X$ is continuous, then one can split this term into $\Xi_{s,t} + R^h_{s,t}$, where $\Xi_{s,t}$ is a germ for the increment of the rough integrals $h(Y^1_t,Y^2_t) - h(Y^1_s,Y^2_s)$ and the remainder term $R^h_{s,t}$ satisfies $R^h_{s,t} \leq \omega(s,t)^{3/p}$ for some regular control function $\omega$. Then, since $3/p >1$, by the uniqueness argument in the sewing lemma, $R^h_{s,t}$ possesses the same bound (up to a universal constant) as the one for $\delta R^h_{sut} = \delta \Xi_{sut} := \Xi_{s,t} - \Xi_{s,u} - \Xi_{u,t}$. Since $ \delta \Xi_{sut}$ is computable, one obtains a bound for the remainder $R^h_{s,t}$ and therefore a bound for the increment $h(Y^1_t,Y^2_t) - h(Y^1_s,Y^2_s) \sim |Y^1_t - Y^2_t| - |Y^1_s - Y^2_s|$, which then allows one to use the rough Gr{\"o}nwall lemma.

  In order to apply this approach for the general case (i.e.~when $\X$ only has c{\`a}dl{\`a}g paths), one needs to invoke the generalized sewing lemma, see e.g.~\cite[Theorem 2.5]{Friz2018}; in particular, one needs to find two (potentially non-regular) controls $\omega_1$, $\omega_2$ such that $|R^h_{s,t}| \leq \omega_1^{\alpha}(s,t-)\omega_2^{\beta}(s+,t)$ with $\alpha + \beta > 1$. Let us consider the same decomposition $h(Y^1_t,Y^2_t) - h(Y^1_s,Y^2_s) = \Xi_{s,t} + R^h_{s,t}$ as in the continuous case. A careful inspection of the rough It\^o formula reveals that $R^h_{s,t}$ contains a term $B_s(X_{s,t}, \mathbb{X}_{s,t})$ for some bilinear form $B_s$ depending only on $\D^2h, Y^1_s, Y^2_s$ and $f$. Clearly, it is a priori only bounded by $\|X\|_{p,[s,t]} \|\mathbb{X}\|_{\frac{p}{2},[s,t]}$ instead of the desired bound $\|X\|_{p,[s,t]}\|\mathbb{X}\|_{\frac{p}{2},(s,t]}$, so that we have to move this term from $R^h_{s,t}$ to the germ $\Xi_{s,t}$. This problem is not present in the continuous case as both terms are equal, since $\X$ has no jumps. As a consequence, the ``c{\`a}dl{\`a}g germ'', denoted by $\tilde{\Xi}_{s,t}$, is much more intricate than the ``continuous germ'' $\Xi_{s,t}$ (keeping in mind that one has to include the jump part arising from the It{\^o} formula into $\tilde{\Xi}_{s,t}$), and therefore the computation of $\delta \tilde{\Xi}_{sut}$ would become very involved.

  This observation shows that the proof methodology based on a rough Gr{\"o}nwall lemma is very difficult to extend to the general case. The situation becomes even more complex when dealing with general time-dependent barriers~$L$ as successively done in \cite{Richard2020}. On the other hand, the approach introduced in the proof of Theorem~\ref{thm:uniqueness RDEs} provides an alternative, relatively simple way to obtain uniqueness of solutions to reflected RDEs, even when jumps are allowed in both the driving rough path $\bX$ and in the barrier~$L$.
\end{remark}


\end{document}